\documentclass[11pt,a4paper]{amsart}

\usepackage{amsmath,amssymb,amsthm,amsfonts}
\usepackage{enumerate,url,mathrsfs,tikz}
\usepackage{hyperref}

\usepackage[utf8]{inputenc}

\numberwithin{equation}{section}
\numberwithin{figure}{section}
\newtheorem{theo}{Theorem}

\newtheorem{lemma}[theo]{Lemma}
\newtheorem{prop}[theo]{Proposition}
\newtheorem{cor}[theo]{Corollary}
\newtheorem{defi}[theo]{Definition}

\newtheorem{exa}[theo]{Example}

\theoremstyle{definition}
\numberwithin{theo}{section}

\newtheorem{rem}[theo]{Remark}

\DeclareMathOperator{\aver}{aver}
\DeclareMathOperator{\diag}{diag}

\DeclareMathOperator{\Id}{Id}

  \def\mG{\mathsf{G}} 
  \def\mH{\mathsf{H}} 
    
  \def\mV{\mathsf{V}}
  \def\mE{\mathsf{E}}

  \def\Ffun{\mathscr{F}}
  \def\mP{\mathsf{P}}
  \def\mK{\mathsf{K}}
  \def\mS{\mathsf{S}}
  \def\mC{\mathsf{C}}

 \def\mv{\mathsf{v}}
 \def\me{\mathsf{e}}
 \def\mw{\mathsf{w}}
 
  \def\mW{\mathsf{W}}
  \def\mf{\mathsf{f}}

\newcommand{\R}{\mathbb{R}}
\newcommand{\N}{\mathbb{N}}
\newcommand{\C}{\mathbb{C}}
\newcommand{\Z}{\mathbb{Z}}

\def\:{\thinspace:\thinspace}
\def\linie{\vrule height 14pt depth 5pt}
\begin{document}

\title{Dynamical systems associated with adjacency matrices} 



\author[D.~Mugnolo]{Delio Mugnolo}

\address{Delio Mugnolo, Lehrgebiet Analysis, Fakult\"at Mathematik und Informatik, Fern\-Universit\"at in Hagen, D-58084 Hagen, Germany}
\email{delio.mugnolo@fernuni-hagen.de}

\thanks{I would like to thank Joachim von Below (Calais), Sylvain Golénia (Bordeaux), James B.\ Kennedy (Lisbon), and Wenlian Lu (Shanghai) for interesting discussions.}

\begin{abstract}
We develop the theory of linear evolution equations associated with the adjacency matrix of a graph, focusing in particular on infinite graphs of two kinds: uniformly locally finite graphs as well as locally finite line graphs. We discuss in detail qualitative properties of solutions to these problems by quadratic form methods. We distinguish between backward and forward evolution equations: the latter have typical features of diffusive processes, but cannot be well-posed on graphs with unbounded degree. On the contrary, well-posedness of backward equations is a typical feature of line graphs. We suggest how to detect even cycles and/or couples of odd cycles on graphs by studying backward equations for the adjacency matrix on their line graph.
\end{abstract}

\maketitle

\section{Introduction} 

The aim of this paper is to discuss the properties of the linear dynamical system
\begin{equation}\label{eq:mainpde}
\frac{du}{dt}(t)=\mathcal Au(t) 
\end{equation}
where $\mathcal A$ is the adjacency matrix of a graph $\mG$ with vertex set $\mV$ and edge set $\mE$. Of course, in the case of finite graphs $\mathcal A$ is a bounded linear operator on the finite dimensional Hilbert space $\C^{|\mV|}$. Hence the solution is given by the exponential matrix $e^{t\mathcal A}$,
but computing it is in general a hard task, since $\mathcal A$ carries little structure and can be a sparse or dense matrix, depending on the underlying graph. Nevertheless, many properties of this matrix can be deduced directly from $\mathcal A$, without knowing $e^{t\mathcal A}$ explicitly; these include order preservation of initial data and symmetry features.

\medskip
However, extending these observations to the case of general infinite graphs is not straightforward, as the symmetric operator $\mathcal A$ is in general neither self-adjoint nor semibounded: e.g., it is known that the non-zero eigenvalues of the adjacency matrix of the star graph $\mS_n$ are  both $\pm\sqrt{n}$ (cf.~\cite[\S~1.4.2]{BroHae12}. In particular, in the limit $n\to \infty$ one cannot hope for either backward or forward well-posedness of~\eqref{eq:mainpde}; furthermore, while the adjacency matrix of a uniformly locally finite graph is self-adjoint (actually, even bounded) by~\cite[Thm.~3.2]{Moh82}, Müller has shown in~\cite{Mue87} that this is generally not the case for generic locally finite graphs -- in fact, he has produced an example of a locally finite graph whose adjacency matrix is not even essentially self-adjoint!

To overcome the problems that derive from unboundedness of the adjacency matrix of infinite graphs, $\mathcal A$ is often normalized: this leads to considering ${\mathcal D}^{-1}\mathcal A$, where $\mathcal D$ is the diagonal matrix whose entries are the vertex degrees of $\mG$.

In this paper we will follow a different approach and restrict to graphs whose adjacency matrices have better properties, namely \textit{uniformly locally finite graphs} and, more interestingly, \textit{line graphs}. The latter are special in that their adjacency matrices have a nice variational structure: we will exploit these features in order to associate $\mathcal A$ with a closed quadratic form and hence with an operator semigroup that will eventually yield the solution to the backward evolution equation associated with~\eqref{eq:mainpde}.

The properties of time-continuous linear dynamical systems on line graphs have been already studied, although seemingly not too often; we mention the interesting paper~\cite{Bob12}, where it is shown that space-continuous diffusion on a quantum graph constructed upon a graph $\mH$ converges towards the space-discrete diffusion on the line graph $\mG$ of $\mH$, provided the transmission conditions in the vertices are suitably tuned.

\medskip
We are not aware of previous investigations specifically devoted to linear dynamical systems associated with $\mathcal A$ like~\eqref{eq:mainpde} apart from the theory of state transfer on graphs, see e.g.~\cite{God12}, which is focused on the unitary group generated by $i\mathcal A$. However, since the  1950s much attentions has certainly been devoted to
\begin{equation}\label{eq:mainpde-lapl}
\frac{du}{dt}(t)=-\mathcal Lu(t)\ , 
\end{equation}
where $\mathcal L$ is the discrete Laplacian of a graph, beginning at the latest with Kato's pioneering discussion about Markovian extensions in~\cite{Kat54}.

Let us denote by $\mathcal D$ the diagonal matrix whose entries are the degrees of the graph's vertices.
 Because $-\mathcal L:=\mathcal A-\mathcal D$ is the discrete counterpart of a free Hamiltonian, $\mathcal A$ can be regarded as a discrete Schrödinger operator that arises from perturbing $-\mathcal L$ by the potential $\mathcal D$: in the case where the graph is $\mathbb Z$ this means that $\mathcal A$ is a special instance of a Jacobi matrix. Jacobi matrices are important models in mathematical physics, beginning with the seminal work~\cite{Ger82} and including the fundamental contribution~\cite{KilSim03}. 

Section~2 is devoted to the study of~\eqref{eq:mainpde} in the easiest case of finite graphs: we will see that the assumption that it takes place on a line graph greatly enriches the theory of its long-time behavior.
In Section~3, we point out the problematic issues that arise when extending the setting from finite to possibly infinite line graphs. We consider a class of 
weighted line graphs and what we regard as the most natural generalization of adjacency matrices in this context: we extend a classical result by Mohar by characterizing boundedness of the adjacency matrix as a linear operator between $\ell^p$-spaces in terms of boundedness of the vertex degree of a line graph or the underlying pre-line graph.
In Section~4 we discuss the forward evolution equation, i.e.,~\eqref{eq:mainpde} for positive time. While this has diffusive nature, and while the discrete Laplacian $\mathcal L$ that governs the classic discrete diffusion equation is positive semi-definite and essentially self-adjoint regardless of the graph, it follows from a result by Golénia that~\eqref{eq:mainpde} enjoys forward well-posedness if and only if $\mG$ is uniformly locally  finite.
Unlike in the case of the classical diffusion equation on domains, manifolds or generic locally finite graphs, we show in Section~5 that~\eqref{eq:mainpde} always enjoys backward well-posedness. This raises the question of the correct interpretation of the forward evolution equation. In Section~6 we propose some comparisons of the semigroup (or group) generated by $\mathcal A$ with other well-known objects of discrete analysis, but the real nature of the backward evolution equation remains elusive.
To conclude, we propose some generalizations on $\mathcal A$ in Section~7: an operator with drift-like term, a quasi-linear version, and an adjacency-like matrix on generalized line graphs. We are going to recall the classical definition of line graphs and the main features of their adjacency matrix in the Appendix.

\medskip
Throughout this paper we are going to impose the following assumptions:
\begin{itemize}
\item $\mG=(\mV,\mE)$ is a locally finite, simple, non-oriented graph;
\item in order to avoid trivialities, both $\mV$ and $\mE$ are assumed to be nonempty.
\end{itemize}

\section{Finite graphs}

As a warm-up, let us  discuss the elementary case of finite graphs $\mG=(\mV,\mE)$. In this case the adjacency matrix $\mathcal A$ is defined by
\[
\mathcal A_{\mv\mw}:=\begin{cases}
1\qquad &\hbox{if $\mv,\mw\in \mV$ are connected by an edge}\\
0 &\hbox{otherwise}.
\end{cases}
\]
Hence, $\mathcal A$
is a finite symmetric matrix and~\eqref{eq:mainpde} is well-posed both backward and forward, i.e.,
the Cauchy problems 
\begin{equation}\label{eq:mainpde-forw}
\begin{cases}
\frac{du}{dt}(t,\mv)&=\mathcal Au(t,\mv), \qquad t\ge 0,\ \mv\in\mV\ ,\\
u(0,\mv)&=u_0(\mv),\qquad \mv\in \mV\ ,
\end{cases}
\end{equation}
and
\begin{equation}\label{eq:mainpde-backw}
\begin{cases}
\frac{du}{dt}(t,\mv)&=\mathcal Au(t,\mv), \qquad t\le 0,\ \mv\in\mV\ ,\\
u(0,\mv)&=u_0(\mv),\qquad \mv\in \mV\ ,
\end{cases}
\end{equation}
enjoy existence and uniqueness of solutions, as well as their continuous dependence on the initial data $u_0$.

\begin{prop}\label{prop:basic-finite}
Let $\mG$ be a finite graph.
Then $\mathcal A$ generates a norm continuous group -- in fact, an analytic group $(e^{z\mathcal A})_{z\in \C}$ -- acting on $\C^{|\mV|}$. The semigroup $(e^{t\mathcal A})_{t\ge 0}$ is positivity preserving, unlike $(e^{t\mathcal A})_{t\le 0}$. Neither of these semigroups is bounded, i.e., $\lim_{t\to \pm \infty}\|e^{t\mathcal A}\|=\infty$.
\end{prop}

\begin{proof}
Due to finiteness of $\mV$, $\mathcal A$ defines a bounded operator on $\C^{|\mV|}$, hence it generates a norm continuous group on it by
\[
e^{z\mathcal A}=\sum_{k=0}^\infty \frac{z^k}{k!}\mathcal A^k\ ,\qquad z\in \C\ .
\]
The semigroup $(e^{t\mathcal A})_{t\ge 0}$ is positivity preserving because its generator has real entries that are positive off-diagonal entries; this does not apply to $-\mathcal A$. 

Because $\mathcal A$ is a zero-trace matrix, its smallest and largest eigenvalues have different signs,
hence neither semigroup can be bounded.
\end{proof}

\begin{rem}
Observe that since $\mathcal A$ has trace 0, by Jacobi's formula we obtain
\[
\det e^{z\mathcal A}=1\quad \hbox{for all }z\in \mathbb C\ ,
\]
hence $(e^{t\mathcal A})_{z\in \C}$ is a non-compact subgroup of the Lie group $SL(\C^{|\mV|})$.
\end{rem}

Since the semigroups $(e^{\pm t\mathcal A})_{t\ge 0}$ are unbounded, neither of them can converge. However, they display interesting behavior up to appropriate rescaling. This is most easily seen in a simple special case: 
If $\mG$ is regular -- say of degree $k$ -- on $n$ vertices, then $k-\mathcal A$ is the discrete Laplacian $\mathcal L$ of $\mG$, while $k+\mathcal A$ is the signless Laplacian $\mathcal Q$, cf.
the surveys~\cite{CveSim09,deLOlideA11}. Accordingly,
\begin{equation}\label{eq:explicit-regular}
e^{t\mathcal A}=e^{kt}e^{-t\mathcal L}\quad\hbox{and} \quad e^{-t\mathcal A}=e^{kt}e^{-t\mathcal Q},\qquad t\ge 0\ .
\end{equation}
Since $0$ is always an eigenvalue of $\mathcal L$ and the associated eigenspace consists of the constant functions on $\mV$, one deduces that
\begin{equation}\label{eq:lta-regular}
\lim_{t\to\infty}e^{-kt}e^{t\mathcal A}=P:=\frac{1}{|\mV|}J_{|\mV|}\ ,
\end{equation}
where $J_n$ is the all-1-matrix of size $n$.
Likewise,
\begin{equation}\label{eq:lta-regular-2}
\lim_{t\to\infty}e^{-kt}e^{-t\mathcal A}=\tilde{P}
\end{equation}
where $\tilde{P}$ is either the orthogonal projector that maps onto functions that take constant value $\pm1$ over each of the two clusters $\mV_\pm$ of a bipartition of $\mV$, or else $0$ if $\mG$ is not bipartite, i.e.,
\[
\tilde{P}:=
\begin{cases}	
\frac{1}{n}\begin{pmatrix}
\lineskip=0pt
J_{|\mV_+|} &\linie & -J_{|\mV_+|\times |\mV_-|}\\
\noalign{\hrule}
-J_{|\mV_-|\times |\mV_+|}  &\linie & J_{|\mV_-|}
\end{pmatrix}\quad&\hbox{if $\mG$ is bipartite with respect to $\mV=\mV_+\dot{\cup}\mV_-$},\\
0 & \hbox{otherwise}
\end{cases}
\]
cf.~\cite[\S~7.8.1]{CveRowSim10}. More precisely, 
\begin{equation}\label{eq:fiedler}
\left\| e^{-tk}e^{t\mathcal A}- P\right\| \le e^{-t\mu_{\min_*}},\qquad t\ge 0
\end{equation}
and
\begin{equation}\label{eq:fiedler-minus}
\left\| e^{-tk}e^{-t\mathcal A}- \tilde{P}\right\| \le e^{-t\nu_{\min_*}},\qquad t\ge 0
\end{equation}
where $\mu_{\min_*}$, $\nu_{\min_*}>0$ denote the lowest non-zero eigenvalue of the discrete and signless Laplacian on $\mG$, respectively. The lowest non-zero eigenvalue $\mu_{\min_*}$  -- the so-called \textit{algebraic connectivity} of $\mG$ -- has been a very popular topic in algebraic graph theory since~\cite{Fie73}; also $\nu_{\min_*}$ has received some attention in latest years, see e.g.~\cite{CarCveRow08}.

Hence, if $\mG$ is regular one can make good use of the rich information available on $\mu_{\min_*},\nu_{\min_*}$ in the literature to describe the long time behavior of both the backward and forward Cauchy problems associated with~\eqref{eq:mainpde}.
In this paper we are  instead going to focus on long-time asymptotics for adjacency matrices of a different class of graphs, namely \textit{line graphs}, cf.\ Section~\ref{sec:append} for some basic definitions and notations.

\begin{rem}\label{rem:linebip}
Both cycle graphs of even length and paths clearly have bipartite line graphs; it follows from Vizing's Theorem~\cite[Thm.~5.3.2]{Die05} that these are in fact the only connected graphs $\mH$ whose line graphs $\mG$ (of course, again paths or even cycles) are bipartite, which makes the information contained in~\eqref{eq:fiedler-minus} pretty much complementary to the main topic of the present paper.
\end{rem}

More generally, the long-time behavior of the relevant semigroups
\[
e^{-t\mathcal A}\quad\hbox{and}\quad e^{t\mathcal A},\qquad t\ge 0
\]
is determined by the smallest and largest eigenvalue $\lambda_{\min}$ and $\lambda_{\max}$ of $\mathcal A$, respectively:
\[
\|e^{-t\mathcal A}\|\le e^{-\lambda_{\min}t}
\quad\hbox{and}\quad
\|e^{t\mathcal A}\|\le e^{\lambda_{\max}t}\ .
\]

Let us look for a more precise description of these semigroups: we deduce from the Spectral Theorem for Hermitian matrices that
\[
e^{-t\mathcal A}=e^{-\lambda_{\min}t}P_{\min}+R_-(t)\quad\hbox{and}\quad e^{t\mathcal A}=e^{\lambda_{\max} t}P_{\max}+R_+(t)\ ,\qquad t\ge 0 \ .
\]
Here $P_{\min}$ and $P_{\max}$ are the eigenprojectors associated with $\lambda_{\min}$ and $\lambda_{\max}$, respectively; only $P_{\max}$ will in general have rank 1.
Furthermore, the remainder terms $R_\pm$ satisfy the estimate
\[
\|R_-(t)\| \le e	^{-t\lambda_{\min_*}}\quad\hbox{and}\quad \|R_+(t)\| \le e	^{t\lambda_{\max_*}}\ :
\]
here $\lambda_{\min_*}$ (resp., $\lambda_{\max_*}$) denotes the second lowest (resp., second largest) eigenvalue of $\mathcal A$.
Accordingly,
\begin{equation}\label{eq:engnag1}
\left\|e^{\lambda_{\min} t}e^{-t\mathcal A}- P_{\min}\right\|\le e^{t(\lambda_{\min}-\lambda_{\min_*})}\ .
\end{equation}
and
\begin{equation}\label{eq:engnag2}
\left\| e^{-\lambda_{\max}t}e^{t\mathcal A}- P_{\max}\right\| \le e^{-t(\lambda_{\max}-\lambda_{\max_*})},\qquad t\ge 0\ ,
\end{equation}

What do $\lambda_{\min},\lambda_{\max}$ and $P_{\min},P_{\max}$ look like in the case of line graphs? 
It is an immediate consequence of~\eqref{eq:variationmain} that all eigenvalues of $\mathcal A$ on the line graph $\mG$ of some $\mH$ are not smaller than $-2$. Furthermore, it has been proved in~\cite[Thm.~4.1]{Doo73} that $-2$ is an eigenvalue of $\mathcal A$ if and only if $\mH$ is neither a tree, nor a unicyclic graph with cycle of odd length; equivalently, if and only if $\mH$ contains (at least) one cycle of even length or (at least) two cycles of odd length.
For the sake of later reference, let us collect some known facts in the following.

\begin{lemma}\label{lemma:basic-spectral-finite}
Let $\mG=(\mV,\mE)$ be the line graph of a simple, connected, finite graph $\mH=(\mV',\mE')$. Then the following assertions hold.

\begin{enumerate}[(1)]
\item The smallest eigenvalue of $\mathcal A$ lies in the interval
\[
[-2,-2\cos(\frac{\pi}{D}+1)]\ ,
\]
where $D$ is the diameter of $\mH$.
\item The smallest eigenvalue is  $-2$ if and only if 
\begin{equation}\label{eq:-2multipl}
|\mE'|-|\mV'|+\beta >0\ ,
\end{equation}
where
\[
\beta :=\begin{cases}
1\quad &\hbox{if $\mH$ is bipartite}\\
0\quad &\hbox{otherwise;}
\end{cases}
\]
in this case, the multiplicity of $-2$ is precisely $|\mE'|-|\mV'|+\beta $.
\item Some $u\in \R^{|\mV|}$ is an eigenvector of $\mathcal A$ for the eigenvalue $-2$ if and only if $\mathcal Ju=0$, where $\mathcal J$ is the (signless) incidence matrix of $\mH$, cf.\ Section~\ref{sec:append}.
\end{enumerate}
\end{lemma}

\begin{proof}
The estimate on $\lambda_{\min}$ is taken from~\cite[Thm.~4.2]{Doo73}, whereas the characterization of the multiplicity of $-2$ as an eigenvalue is a consequence of known properties of the incidence matrix and the Rank-Nullity-Theorem stated e.g.\ in~\cite[Thm.~2.2.4]{CveRowSim04}.
\end{proof}

\begin{rem}\label{rem:estim:collatz}
On the other hand, the largest eigenvalue $\lambda_{\max}$ satisfies 
\begin{equation}\label{eq:estim:collatz}
\left.
\begin{array}{r}
\deg_{\aver}(\mG)\\[5pt]
\sqrt{\deg_{\max}(\mG)}\\[5pt]
2\cos\frac{\pi}{|\mV|+1}
\end{array}
\right\}\le \lambda_{\max}< \deg_{\max}(\mG)\ ;
\end{equation}
unless $\mG$ is $k$-regular, in which case $\lambda_{\max}=k$, cf.~\cite[Prop.~3.1.2]{BroHae12},~\cite[Lemma~3.3]{DvoMoh10}, and~\cite[Satz 2]{ColSin57}.
(Here we denote by $\deg_{\aver}$ the \textit{average degree} of $\mG$, i.e., $2\frac{|\mE|}{|\mV|}$, whereas $\deg_{\aver}$ is the \textit{maximal degree}.)

Observe that in view of strict monotonicity of $\lambda_{\max}$ under edge deletion,  $\lambda_{\max}$ can only agree with $\sqrt{\deg_{\max}(\mG)}
$ if $\mG$ is a star, which is impossible whenever $\mG$ is a line graph. In the special case of line graphs, an improvement to the second lower bound has been suggested in~\cite{Mey17-mo}.
\end{rem}

\begin{exa}
The complete graph $\mK_4$ is the line graph of the star graph on $4$ edges, hence~\eqref{eq:-2multipl} is not satisfied.
The diamond graph (i.e., $\mK_4$ minus an edge) is the line graph of the paw graph (i.e., $\mK_3$ plus a pendant edge), for which~\eqref{eq:-2multipl} is not satisfied.
Hence, in both cases $-2$ is not an eigenvalue of $\mathcal A$ and accordingly $e^{2t}e^{-t\mathcal A}\to 0$ as $t\to \infty$.

Let $\mH$ be a graph consisting of two independent cycles (of possibly different length) joined at exactly one vertex. Then $-2$ is an eigenvalue of $\mathcal A(\mG)$, regardless of the length of the two cycles: its multiplicity is $2$ if both cycles have even length and 1 otherwise.
\end{exa}

\begin{cor}
Under the assumptions of Lemma~\ref{lemma:basic-spectral-finite} the matrix $\mathcal A$ has eigenvalue $-2$ if and only if $\mH$ contains at least one cycle of even length (and in this case also $2$ is an eigenvalue) or at least two cycles of odd length.

Furthermore, $-2$ is a \textit{simple} eigenvalue of $\mathcal A$ if and only if $\mH$ either is non-bipartite and contains exactly two independent cycles, or is bipartite and unicyclic.
\end{cor}

\begin{cor}\label{cor:fiedl}
Under the assumptions of Lemma~\ref{lemma:basic-spectral-finite} and denoting by $P_{- 2}$ the orthogonal projectors onto the eigenspaces for the eigenvalue $- 2$, the following assertions hold.

\begin{enumerate}
\item If $\mH$ contains at least one cycle of even length or at least two cycles of odd length, then 
\[
\left\| e^{-2 t}e^{-t\mathcal A(\mG)}- P_{-2} \right\| \le e^{-t\epsilon},\qquad t\ge 0
\]
for some $\epsilon>0$. The projector $P_{-2}$ has rank 1 if and only if $\mG$ contains precisely one cycle of even length or precisely two cycles of odd length.

\item If $\mH$ is a tree or a unicyclic graph with odd cycle, then
\[
\lim_{t\to \infty} e^{-2t}e^{-t\mathcal A(\mG)}=0\ .
\]

\item If $\mG$ is not a regular graph, then 
\[
\lim_{t\to \infty} e^{-t\deg_{\max}(\mG)}e^{t\mathcal A(\mG)}=0\ .
\]
\end{enumerate}
\end{cor}

The long-time behavior described in 2) and 3) for the adjacency matrix of line graphs should be compared with~\eqref{eq:lta-regular} and~\eqref{eq:lta-regular-2}, which holds for the adjacency matrix of regular graphs instead.

\begin{proof}
(1) By Lemma~\ref{lemma:basic-spectral-finite}, $2$ is the lowest eigenvalue of $\mathcal A$ since $|\mE'|-|\mV'|+\beta>0$. Now the assertion follows from~\eqref{eq:engnag1} since $-2$ is the dominant eigenvalue of $-\mathcal A$ by Lemma~\ref{lemma:basic-spectral-finite}.

(2) Under these assumptions, all eigenvalues of $\mathcal A$ are strictly larger than $-2$.

(3) The assertion follows from~\eqref{eq:engnag2}, since by Lemma~\ref{lemma:basic-spectral-finite}.(3) the dominant eigenvalue of $\mathcal A$ is strictly smaller than $\deg_{\max}(\mG)$ unless $\mG$ is regular.
\end{proof} 

Unfortunately, the above asymptotic assertions are of little use as long as the eigenprojectors of $\mathcal A$ -- and in particular its Perron eigenvector, unless $\mG$ is regular -- are unknown; path graphs are among the few classes of non-regular graphs whose Perron eigenvector is known, cf.~\cite[\S~1.4.4]{BroHae12}.

Things look better when it comes to the least eigenvalue $\lambda_{\min}$ of $-\mathcal A$. If $\mH$ contains an even cycle, then we already known that $-2$ is an eigenvalue of $\mathcal A$ and the following can be checked directly.

\begin{prop}\label{prop:accord}
Under the assumptions of Lemma~\ref{lemma:basic-spectral-finite}, let $\mH$ contain an even cycle $\mH_0$ as an induced subgraph and consider its edge bipartition into $\mE'_0=\mE'_+\dot{\cup}\mE'_-$; this  corresponds to the vertex bipartition $\mV_0:=\mV_+\dot{\cup}\mV_-$  of $\mG$. Then
\[
u(\mv):=\begin{cases}
1\quad&\hbox{if }\mv\in \mV_+\\
-1 & \hbox{if }\mv\in \mV_-\\
0 & \hbox{otherwise}
\end{cases}
\]
is an eigenvector of $\mathcal A(\mG)$ for the eigenvalue $-2$.
\end{prop}

\begin{rem}
1) Proposition~\ref{prop:accord} suggests that the backward evolution equation~\eqref{eq:mainpde} can be used to detect even cycles in line graphs.
For instance, if $\mH$ is a Hamiltonian graph on an even number of vertices, its Hamiltonian cycles may then be detected by considering the long-time behavior of~\eqref{eq:mainpde} on its line graph $\mG$. More generally, a necessary and sufficient condition for a graph with an even number of vertices to be Hamiltonian is that the adjacency matrix of its line graph has eigenvalue $-2$ and that the associated eigenprojector is irreducible.

2) If $\mH$ contains pairs of odd cycles, the eigenvectors corresponding to $\lambda_{\min}=-2$ have been described by algebraic means in~\cite[Thm.~2.1']{Doo73} and~\cite[\S~2.6]{CveRowSim04}, but it does not seem to be straightforward to find a formula for the associated eigenprojectors. 
\end{rem}

\begin{exa}\label{exa:mainexaasymp}
If $-2$ is an eigenvalue of $\mathcal A$, then $(e^{2t}e^{t\mathcal A(\mG)})_{t\le 0}$ converges as $t\to -\infty$ towards an eigenprojector that conveys information on the combinatorics of the pre-line graph $\mH$. Let us consider a few cases displaying this principle.

\begin{enumerate}
\item  Let $\mH$ be the graph depicted in Figure~\ref{fig:firstfig} and $\mG$ its line graph.

\begin{figure}
\begin{minipage}{4.5cm}
\begin{tikzpicture}[scale=0.8]
\draw (4.6,-1) -- (4.6,1) -- (6,0) -- cycle;
\draw (8,0) -- (9.4,-1) -- (9.4,1) -- cycle;
\draw (6,0) -- (8,0);
\draw[fill] (4.6,-1) circle (2pt);
\draw[fill] (4.6,1) circle (2pt);
\draw[fill] (6,0) circle (2pt);
\draw[fill] (8,0) circle (2pt);
\draw[fill] (9.4,-1) circle (2pt);
\draw[fill] (9.4,1) circle (2pt);
\draw[fill] (4.6,0)  node[left=1pt] {$\me_1$};
\draw[fill] (5.6,.5) node[above=1pt] {$\me_2$};
\draw[fill] (5.6,-.5) node[below=1pt] {$\me_3$};
\draw[fill] (7,0) node[above=1pt] {$\me_4$};
\draw[fill] (8.4,.5)  node[above=1pt] {$\me_5$};
\draw[fill] (8.4,-.5) node[below=1pt] {$\me_6$};
\draw[fill] (9.4,0)   node[right=1pt] {$\me_7$};
\end{tikzpicture}
\end{minipage}\qquad
\begin{minipage}{4.5cm}
\begin{tikzpicture}[scale=0.8]
\draw (4.6,-1) -- (4.6,1) -- (3.2,0) -- cycle;
\draw (4.6,-1) -- (4.6,1) -- (6,0) -- cycle;
\draw (6,0) -- (7.4,-1) -- (7.4,1) -- cycle;
\draw (8.8,0) -- (7.4,-1) -- (7.4,1) -- cycle;
\draw[fill] (4.6,-1) circle (2pt) node[below=1pt] {$\mv_3$};
\draw[fill] (4.6,1) circle (2pt) node[above=1pt] {$\mv_2$};
\draw[fill] (6,0) circle (2pt) node[above=1pt] {$\mv_4$};
\draw[fill] (7.4,-1) circle (2pt) node[below=1pt] {$\mv_6$};
\draw[fill] (7.4,1) circle (2pt) node[above=1pt] {$\mv_5$};
\draw[fill] (3.2,0) circle (2pt) node[above=1pt] {$\mv_1$};
\draw[fill] (8.8,0) circle (2pt)  node[above=1pt] {$\mv_7$};
\end{tikzpicture}
\end{minipage}
\caption{The  graph $\mH$ and its line graph $\mG$ in Example~\ref{exa:mainexaasymp}.1); here $\mv_i\simeq \me_i$}\label{fig:firstfig}
\end{figure}

The graph $\mH$ has two odd cycles, hence by Lemma~\ref{lemma:basic-spectral-finite} $-2$ is a simple eigenvalue of the adjacency matrix of $\mG$. The associated eigenvector is
\[
(1,-1,-1,2,-1,-1,1)
\]
and the corresponding eigenprojector is
\[
P_{-2}:=\frac{1}{10}\begin{pmatrix}
1 & -1 & -1 & 2 & -1 & -1 & 1\\
-1 & 1 & 1 & -2 & 1 & 1 & -1\\
-1 & 1 & 1 & -2 & 1 & 1 & -1\\
2 & -2 & -2 & 4 & -2 & -2 & 2\\
-1 & 1 & 1 & -2 & 1 & 1 & -1\\
-1 & 1 & 1 & -2 & 1 & 1 & -1\\
1 & -1 & -1 & 2 & -1 & -1 & 1\\
\end{pmatrix}
\]

\item  Let $\mH$ be a cycle of length 4 with an edge appended to one of its vertices.
\begin{figure}
\begin{minipage}{5cm}
\begin{tikzpicture}[scale=0.8]
\draw (0,0) -- (2,2) -- (4,0) -- (2,-2) -- cycle;
\draw (4,0) -- (6,0);
\draw[fill] (0,0) circle (2pt);
\draw[fill] (2,2) circle (2pt);
\draw[fill] (2,-2) circle (2pt);
\draw[fill] (4,0) circle (2pt);
\draw[fill] (6,0) circle (2pt);
\draw[fill] (.9,-1)  node[below=1pt] {$\me_2$};
\draw[fill] (.9,1) node[above=1pt] {$\me_1$};
\draw[fill] (3.1,-1) node[below=1pt] {$\me_4$};
\draw[fill] (3.1,1) node[above=1pt] {$\me_3$};
\draw[fill] (5,0) node[above=1pt] {$\me_5$};
\end{tikzpicture}
\end{minipage}\quad
\begin{minipage}{4.5cm}
\begin{tikzpicture}[scale=0.8]
\draw (0,-1) -- (2,-1) -- (2,1) -- (0,1) -- cycle;
\draw (2,1) -- (3.5,0) -- (2,-1);
\draw[fill] (0,-1) circle (2pt) node[below=1pt] {$\mv_2$};
\draw[fill] (0,1) circle (2pt) node[above=1pt] {$\mv_1$};
\draw[fill] (2,-1) circle (2pt) node[below=1pt] {$\mv_4$};
\draw[fill] (2,1) circle (2pt) node[above=1pt] {$\mv_3$};
\draw[fill] (3.5,0) circle (2pt) node[above=1pt] {$\mv_5$};
\end{tikzpicture}
\end{minipage}
\caption{The  graph $\mH$ and its line graph $\mG$ in Example~\ref{exa:mainexaasymp}.2); again, $\mv_i\simeq \me_i$}\label{fig:firstfig-2}
\end{figure}
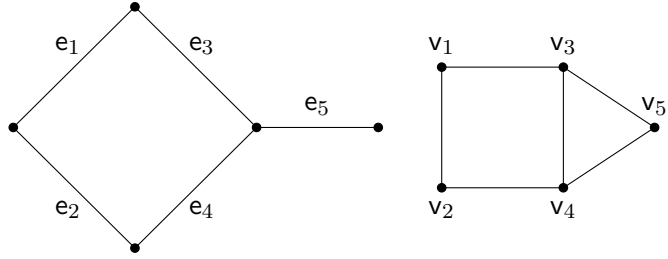

The graph $\mH$ has one even cycle, hence by Lemma~\ref{lemma:basic-spectral-finite} $-2$ is a simple eigenvalue of the adjacency matrix of $\mG$. The associated eigenvector is
\[
(1,-1,-1,1,0)
\]
and the corresponding eigenprojector is
\[
P_{-2}:=\frac{1}{4}
\begin{pmatrix}
1 & -1 & -1 & 1 & 0\\
-1 & 1 & 1 & -1 & 0\\
-1 & 1 & 1 & -1 & 0\\
1 & -1 & -1 & 1 & 0\\
0 & 0 & 0 & 0 & 0\\
\end{pmatrix}\ .
\]

\item If $\mH$ is the figure-8-graph consisting of two cycles of length 4 each, $-2$ is an eigenvalue of multiplicity 2 of the adjacency matrix of $\mathcal A$; the associated eigenspace is spanned by the vectors
\[
(0,0,0,0,1,-1,-1,1)\quad\hbox{and}\quad (1,-1,-1,1,0,0,0,0)
\]
hence, the semigroup $(e^{-t(2+\mathcal A)})_{t\ge 0}$ converges towards the orthogonal projector
\[
P_{-2}:=\frac{1}{4}\begin{pmatrix}
1 & -1 & -1 & 1 & 0 & 0 & 0 & 0\\
-1 & 1 & 1 & -1 & 0 & 0 & 0 & 0\\
-1 & 1 & 1 & -1 & 0 & 0 & 0 & 0\\
1 & -1 & -1 & 1 & 0 & 0 & 0 & 0\\
0 & 0 & 0 & 0 & 1 & -1 & -1 & 1\\
0 & 0 & 0 & 0 & -1 & 1 & 1 & -1\\
0 & 0 & 0 & 0 &-1 & 1 & 1 & -1\\
0 & 0 & 0 & 0 & 1 & -1 & -1 & 1
\end{pmatrix}\ .
\]

Observe that this eigenvector does \textit{not} correspond to the cycle of length 4 in the diamond graph -- since this is not bipartite -- but to the two independent cycles of length 3.

\item If $\mH$ is the figure-8-graph consisting of one cycle of length 4 and one cycle of length 3, $-2$ is an eigenvalue of multiplicity 1 of the adjacency matrix of $\mG$; the associated eigenspaces is spanned by the vector
\[
(1,-1,-1,1,0,0,0,0)
\]
hence $(e^{-t(2+\mathcal A)})_{t\ge 0}$ converges towards the orthogonal projector
\[
P_{-2}:=\frac{1}{4}\begin{pmatrix}
1 & -1 & -1 & 1 & 0 & 0 & 0 & 0\\
-1 & 1 & 1 & -1 & 0 & 0 & 0 & 0\\
-1 & 1 & 1 & -1 & 0 & 0 & 0 & 0\\
1 & -1 & -1 & 1 & 0 & 0 & 0 & 0\\
0 & 0 & 0 & 0 & 0 & 0 & 0 & 0\\
0 & 0 & 0 & 0 & 0 & 0 & 0 & 0\\
0 & 0 & 0 & 0 & 0 & 0 & 0 & 0\\
0 & 0 & 0 & 0 & 0 & 0 & 0 & 0\\
0 & 0 & 0 & 0 & 0 & 0 & 0 & 0\\
	\end{pmatrix}
\]

\item  By Beineke's Theorem, the wheel graph $\mW_n$ on $n+1$ vertices is a line graph if and only $n\in\{3,4\}$; we have already considered the case $\mW_3=\mK_4$. Instead, $\mW_4$ is the line graph of the diamond graph, which has $5$ edges, $4$ vertices and is not bipartite; accordingly, $-2$ is a simple eigenvalue of $\mathcal A(\mW_4)$.
The associated eigenspace is spanned by the vector
\[
(0,1,-1,1,-1)
\]
hence $(e^{-t(2+\mathcal A)})_{t\ge 0}$ converges towards the orthogonal projector
\[
P_{-2}:=\frac{1}{4}
\begin{pmatrix}
0 & 0 & 0 & 0 & 0\\
0 &1&-1&1&-1\\
 0 &-1&1&-1&1\\
 0 &1&-1&1&-1\\
0 &-1&1&-1&1
 \end{pmatrix}
 \]
\end{enumerate}
\end{exa}

\begin{exa}
It seems more difficult to make combinatorial sense of the long-time behavior of the forward equation.
\begin{enumerate}

\item Cycles of odd length are 2-regular and not bipartite, hence we already know from~\eqref{eq:lta-regular} and~\eqref{eq:lta-regular-2} that $(e^{2t}e^{t\mathcal A})_{t\le 0}$ converges towards the orthogonal projector onto the constant functions over $\mG$ as $t\to-\infty$, while $(e^{2t}e^{t\mathcal A})_{t\ge 0}$ converges towards 0 as $t\to\infty$. Choosing a more appropriate rescaling, one finds for $\mG=\mC_3$ and $\mG=\mC_5$ the asymptotic behavior
\[\frac{e^{t}}{3}
\begin{pmatrix}
2 & -1 & -1\\
-1 & 2 & -1\\
-1 & -1 & 2
\end{pmatrix}
\quad \hbox{and}\quad 
\frac{e^{\frac{t}{2}(\sqrt{5}+1)}}{10}
\begin{pmatrix}
4 & -\sqrt{5}-1 & \sqrt{5}-1 & \sqrt{5}-1 & -\sqrt{5}-1\\
-\sqrt{5}-1 & 4 & -\sqrt{5}-1 & \sqrt{5}-1 & \sqrt{5}-1 \\
\sqrt{5}-1 & -\sqrt{5}-1 & 4 & -\sqrt{5}-1 & \sqrt{5}-1 \\
 \sqrt{5}-1 & \sqrt{5}-1 & -\sqrt{5}-1 & 4 & -\sqrt{5}-1 \\
 -\sqrt{5}-1 & \sqrt{5}-1 & \sqrt{5}-1 & -\sqrt{5}-1 & 4 
\end{pmatrix}
\]
as $t\to \infty$, respectively.

\item Let $\mG=\mW_4$. Then $(e^{t\mathcal A(\mW_4)})_{t\ge 0}$ has the asymptotic behavior
 \[
\frac{e^{t(\sqrt{5}+1)}}{10}\begin{pmatrix}
(5-\sqrt{5}) & -\sqrt{5} & -\sqrt{5} & -\sqrt{5} & -\sqrt{5}\\
 -\sqrt{5} &(5-\sqrt{5}) & -\sqrt{5} & -\sqrt{5} & -\sqrt{5}\\
 -\sqrt{5} & -\sqrt{5} & (5-\sqrt{5}) & -\sqrt{5} & -\sqrt{5}\\
 -\sqrt{5} & -\sqrt{5} & -\sqrt{5} & (5-\sqrt{5}) & -\sqrt{5}\\
 -\sqrt{5} & -\sqrt{5} & -\sqrt{5} & -\sqrt{5} & (5-\sqrt{5}) 
 \end{pmatrix} 
\]
for $t\to \infty$.
\end{enumerate}
\end{exa}

\begin{prop}
Under the assumptions of Lemma~\ref{lemma:basic-spectral-finite}, let $p\in [2,\infty]$. Then the following assertions hold for the norms of $(e^{t\mathcal A})_{t\ge 0}$ regarded as a semigroup of (bounded) linear operators acting on $\ell^p(\mV)$.

1) The backward semigroup  satisfies
\[
\|e^{t\mathcal A} \|_{\mathcal L(\ell^p(\mV))}\le e^{-t(\frac{4}{p}+\frac{p-2}{p}\deg_{\max}(\mG))}\quad\hbox{for all }t\le 0\ .
\]
In fact, 
\[
\lim_{t\to-\infty}  e^{t(\frac{4}{p}+\frac{p-2}{p}\deg_{\max}(\mG))}\|e^{t\mathcal A} \|_{\mathcal L(\ell^p(\mV))}= 0\
\]
if $\mH$ is either a tree or a unicyclic graph with cycle of odd length.

2) The forward semigroup  satisfies
\[
\|e^{t\mathcal A} \|_{\mathcal L(\ell^p(\mV))}\le e^{t\deg_{\max}(\mG)}\quad\hbox{for all }t\ge 0\ .
\]
In fact, 
\[
\lim_{t\to\infty} e^{-t\deg_{\max}(\mG)}\|e^{t\mathcal A}\|_{\mathcal L(\ell^p(\mV))}= 0
\]
unless $\mG$ is regular.

\end{prop}

Analogous assertions can be formulated for $p\in [1,2]$ by duality.

\begin{proof}
1) The matrix $-(\mathcal A+2)$ is negative semidefinite (definite if $\mH$ is a tree or a unicyclic graph with odd cycle) in $\ell^2(\mV)$ by Lemma~\ref{lemma:basic-spectral-finite} and $(e^{-t(\mathcal A+\deg_{\max}(\mG))})_{t\ge 0}$ is contractive with respect to the $\ell^\infty$-norm by~\cite[Lemma~6.1]{Mug07}. Hence, by the Interpolation Theorem of Riesz--Thorin, one finds
\[
\|e^{-t\mathcal A}\|_{\mathcal L(\ell^p(\mV))}  \le \|e^{-t\mathcal A}\|^\frac{2}{p}_{\mathcal L(\ell^2(\mV))} \|e^{-t\mathcal A}f\|^\frac{p-2}{p}_{\mathcal L(\ell^\infty(\mV))}\le e^{t\frac{4}{p}} e^{t\frac{p-2}{p} \deg_{\max}(\mG)}\ \hbox{for all } t\ge 0\ .
\]

Analogous considerations yield 2): consider the perturbed matrix $\mathcal A-\deg_{\max}(\mG)$. Then, $\mathcal A-\deg_{\max}(\mG)$ is negative semi-definite in $\ell^2(\mV)$ by Remark~\ref{rem:estim:collatz}, hence
\[
\|e^{t\mathcal A}f\|_2 \le e^{t\deg_{\max}(\mG)}\|f\|_2 \quad\hbox{for all }t\ge 0\ .
\]
Furthermore, by~\cite[Lemma~6.1]{Mug07} $(e^{t(\mathcal A-\deg_{\max}(\mG))})_{t\ge 0}$ is contractive with respect to the $\ell^\infty$-norm, i.e.,
\[
\|e^{t\mathcal A}f\|_\infty \le e^{t\deg_{\max}(\mG)}\|f\|_\infty \quad\hbox{for all }t\ge 0\ .
\]
The assertion now follows from the Riesz-Thorin Theorem.
\end{proof}

\begin{prop}
The semigroup generated by $\mathcal A$ satisfies
\[
\sum_{\mv\in \mV}e^{t\mathcal A} f (\mv)\le e^{t\deg_{\max}(\mG)}\sum_{\mv\in \mV} f (\mv)\qquad \hbox{for all $0\le f\in \R^{|\mV|}$ and all $t\ge 0$.}
\]
\end{prop}

\begin{proof}
By Proposition~\ref{prop:basic-finite} we already know that $(e^{t\mathcal A})_{t\ge 0}$ and hence $(e^{t(\mathcal A-\deg_{\max})})_{t\ge 0}$ are positive semigroups. Let us also observe that $(\mathcal A-\deg_{\max}){\bf 1}\le 0$, hence $(e^{t(\mathcal A-\deg_{\max})})_{t\ge 0}$ is a substochastic semigroup, whence the claim follows.
\end{proof}

\section{Infinite graphs}\label{section:infinite}

In the following, we are going to study some properties of $\mathcal A$ on countably infinite graphs.

A well-known result by Mohar states that the adjacency matrix
$\mathcal A$ of a locally finite graph $\mG$ is a bounded operator on $\ell^p(\mV)$ for some/all $p\in [1,\infty]$ if and only if $\mG$ is uniformly locally finite (meaning that the vertex degrees of $\mG$ are uniformly bounded), cf.~\cite[Thm.~3.2]{Moh82}. This is in turn equivalent to uniform local finiteness of the pre-line graph $\mH$.

\begin{lemma}\label{rem:degree}
A graph $\mH$ is uniformly locally finite if and only if its line graph $\mG$ is uniformly locally finite.
\end{lemma}

\begin{proof}
Given a vertex $\mv\simeq \me'\simeq (\mv',\mw')\in \mV$, its degree $\deg_{\mG}(\mv)$ in $\mG$ is given by
\begin{equation}\label{eq:eqdegreee}
\deg_{\mG}(\mv)=\deg_{\mH}(\mv')+\deg_{\mH}(\mw')-2 \ ,
\end{equation}
hence in particular $\deg_{\max}(\mG)\le 2\deg_{\max}(\mH)-2$ and accordingly $\mG$ is uniformly locally finite if so is $\mH$. Conversely, if $\mH$ is not uniformly locally finite, we can find a sequence $(\mv'_n)_{n\in \mathbb N}\subset \mV'$ such that $\deg_\mH(\mv'_n)\ge n$. For each $n$ we can thus pick some $\me'_n$ incident in $\mv'_n$ whose degree is then by~\eqref{eq:eqdegreee} at least $n-1$: upon possibly discarding edges considered twice, this procedure yields a sequence $(\me'_n)_{n\in \N}\simeq (\mv_n)_{n\in \mathbb N} \subset \mV$ with $\deg_\mG(\mv_n)\to \infty$. 
\end{proof}

For later convenience, let us state Mohar's result in a slightly more general form.

\begin{lemma}\label{lem:mohar}
Let $\mG$ be a  graph whose edges are weighted by $0<c\in \ell^\infty(\mE)$.
Then the adjacency matrix formally defined by
\begin{equation}\label{def:acnonline}
(\mathcal A_C)_{\mv\mw}=\begin{cases}
0\quad &\hbox{if either $\mv=\mw$ or $\mv,\mw$ are not adjacent}\\
c(\me)&\hbox{if $\mv,\mw$ are the endpoints of $\me$}
\end{cases}
\end{equation}
is bounded on $\ell^p(\mV)$ for some/all $p\in [1,\infty]$ if and only if $\deg_\mG^C\in \ell^\infty(\mV)$, where
\begin{equation}\label{eq:def-degc}
\deg_\mG^C(\mv):=\sum_{\me\sim \mv}c(\me)
\end{equation}
and in this case its operator norm satisfies
\begin{equation}\label{eq:gerschgorin-reloaded}
\left(\sup\limits_{\mv\in \mV}\sum_{\me\sim\mv}c(\me)\right)^\frac{1}{p}\le \|\mathcal A_C\|_{\mathcal L(\ell^p(\mV))}\le \sup\limits_{\mv\in \mV}\sum_{\me\sim\mv}c(\me)\ \quad \hbox{for all }p\in [1,\infty]\ ,
\end{equation}
with equality in the upper bound holding for $p=1$ and $p=\infty$. The lower bound follows as in~\cite[Thm.~3.1]{Bel09}.
\end{lemma}

We stress that Lemma~\ref{lem:mohar} holds whether $\mG$ is a line graph or not. Furthermore, since $\mG$ is by assumption simple and hence there is at most one edge between any two vertices, the weight function $c$ can be regarded as a mapping from $\mV\times \mV$ to $(0,\infty)$, i.e., $\me\equiv (\mv,\mw)$; however in view of our standing assumption that $\mG$ is non-oriented, the weight functions is necessarily symmetric, i.e., $c(\mv,\mw)=c(\mw,\mv)$.

\begin{proof}
Since $\mathcal A_C$ is an infinite matrix with nonnegative entries, its boundedness as an operator on $\ell^p(\mV)$ is equivalent to the property of being defined everywhere, which is in turn easily seen to be equivalent to the property of its rows being uniformly bounded, i.e., to $(\deg_{\mG}^C(\mv))_{\mv\in \mV}\in \ell^\infty(\mV)$. Formula~\eqref{eq:gerschgorin-reloaded} follows from~\cite[Thm.~2.2.5 and Thm.~2.2.8]{Dav07} for $p=1$ and $p=\infty$, respectively;
the case of general $p$ can be deduced by the Riesz--Thorin Theorem.
\end{proof}

Whenever $\mG$ is a line graph, cf.\ Section~\ref{sec:append}, it is more convenient and perhaps more natural to restrict to a special class of edge weights $c$ that depend on the structure of the underlying pre-line graph $\mH=(\mV',\mE')$. We thus impose in the remainder of this paper that
\begin{equation}\label{eq:standing}
c:\mV'\to \R \hbox{ s.t. }C_0 \ge c(\mv')\ge c_0 \hbox{ for some }C_0>c_0>0\hbox{ and all }\mv'\in \mV'.
\end{equation}
Let us also recall that $\mathcal J$ is the (signless) incidence matrix of $\mG$, cf.~\eqref{eq:def-signlessinc}.

\begin{prop}\label{prop:mohar-reloaded}
The following assertions hold for the matrix $\mathcal A_C=(a_{C_{\mv\mw}})_{\mv,\mw\in \mV}$ associated with the quadratic form
\[
a_C(u):=\sum_{\mv'\in\mV'} c(\mv')|\mathcal J u(\mv')|^2-\sum_{\mv\in\mV}\gamma(\mv)|u(\mv)|^2
\]
with (form) domain
\[
D^N(a):=\{u\in \ell^2(\mV):\mathcal Ju\in \ell^2(\mV')\}\ ,
\]
where
\[
\gamma(\mv):=c(\mv')+c(\mw')\quad \hbox{if }\mv=\mw\simeq \me'=(\mv',\mw')\ .
\]

1) The matrix $\mathcal A_C$ is a self-adjoint operator on $\ell^2(\mV)$ that is bounded from below.

2) Furthermore, $\mathcal A_C$ is bounded from above (and hence is bounded) on $\ell^p(\mV)$ for some/all $p\in [1,\infty]$ if and only if $\deg_{\mG}^C(\mv)\in \ell^\infty(\mV)$, where
\[\deg_{\mG}^C(\mv):= c(\mv')(\deg_\mH (\mv')-1)+c(\mw')(\deg_\mH (\mw')-1) \quad\hbox{whenever }\mv\simeq \me'\equiv (\mv',\mw')\in \mE'\ ;
\]
in this case \begin{equation}\label{eq:gerschgorin-reloaded-2}
\|\mathcal A_C\|\le\|\deg_{\mG}^C\|_\infty\ .
\end{equation}

3) If $\deg_\mG^C\in \ell^\infty(\mG)$, then the spectrum $\sigma(\mathcal A_C)$ of $\mathcal A$ as an operator on $\ell^2(\mV)$ is contained in the interval
\[
[-\|\gamma\|_\infty ,\|\deg_{\mG}^C\|_\infty \ ]\ .
\]

4) The spectral radius
\[
s_p(\mathcal A_C):=\sup\{\lambda:\lambda\in \sigma(\mathcal A_C)\}
\]
of $\mathcal A_C$ as a bounded linear operator on $\ell^p(\mV)$
satisfies
\[
\left(\|\deg_{\mG}^C\|_\infty\right)^\frac{1}{p}\le s_p(\mathcal A_C)\le \|\deg_{\mG}^C\|_\infty\ .
\]
\end{prop}

\begin{proof}
1) With the notation of Section~\ref{sec:append}, it is known that $\mathcal J$ is always densely defined since so are the positive and negative parts $\mathcal I^\pm$ of $\mathcal I$, cf.~\cite[Prop.~2.4 and Rem.~2.5]{Car11}: hence $a_C$ is densely defined.

Due to our assumptions
on $\mG$ and $C$
 the sequence $\left(\gamma(\mv)\right)_{\mv\in \mV}$ lies in $\ell^\infty(\mV)$. Thus, the symmetric form $a_C$ is $\ell^2(\mV)$-elliptic and bounded: indeed, the weighted $\ell^2$-norm $|\!\|u\|\!|:=\sum_{\mv\in\mV}|u(\mv)|^2 \gamma(\mv)$ is equivalent to the standard one on $\ell^2(\mV)$. In other words, $a_C$ is a closed quadratic form, hence it is associated with a self-adjoint operator $\mathcal A_C$ that is bounded from below.

2) The assertion is a consequence of Lemma~\ref{lem:mohar}, since a direct computation shows that $\mathcal A_C$ acts on functions with finite support as the matrix defined by
\begin{equation}\label{eq:adjweighted}
(\mathcal A_C)_{\mv\mw}=\begin{cases}
0\quad &\hbox{if either $\mv=\mw$ or $\mv,\mw$ are not adjacent},\\
c(\mv')&\hbox{if $\mv,\mw$ are distinct but adjacent vertices in $\mG$ (i.e., edges in its pre-line graph $\mH$)}\\
&\quad \hbox{and $\mv'\in\mV'$ is their common endpoint},
\end{cases}
\end{equation}
i.e.,
\begin{equation}\label{Ac00}
\mathcal A_C f(\mv)=\sum_{\mw\stackrel{\mv'}{\sim}\mv}c(\mv')f(w),\qquad \hbox{for all }f\in c_{00}(\mV)\hbox{ and }\mv\in\mV\ .
\end{equation}
(Here and in the following we denote by $c_{00}(\mV)$ the space of finitely supported functions from $\mV$ to $\C$.)

3) The upper and lower bound on the spectrum follows directly from~\eqref{eq:gerschgorin-reloaded} and the definition of $a_C$, respectively. The estimate on the spectral radius can be proved as in~\cite[Thm.~3.1]{Bel09}.
\end{proof}

\begin{cor}
Under the assumptions of Proposition~\ref{prop:mohar-reloaded}, $-\mathcal A_C$ generates a semigroup of self-adjoint, quasi-contractive analytic semigroup $(e^{-t\mathcal A_C})_{t\ge 0}$ on $\ell^2(\mV)$, which is additionally norm continuous -- and hence extends to a group $(e^{z\mathcal A_C})_{z\in \C}$ -- if and only if $\deg_{\mG}^C\in \ell^\infty(\mV)$.
\end{cor}

\begin{exa}\label{exa:Z}
1) Let us consider the simplest infinite graph: the graph whose vertex set is $\Z$ and such that any two vertices are adjacent if and only if their difference is $\pm1$. This graph is regular and therefore 
its adjacency matrix is a bounded operator on $\ell^2(\mV)\equiv \ell^2(\Z)$, hence it generates a semigroup of bounded linear operators that can luckily be written down explicitly: for $f\in \ell^2(\mV)$ we have
\begin{equation}\label{eq:estliz}
e^{-t\mathcal L(\Z)}f(\mv)=\sum_{\mw\in \Z}f(\mw) \frac{1}{2\pi}\int_{-\pi}^\pi \cos((\mv-\mw)q) e^{-2t(1-\cos q)}dq=
\sum_{\mw\in \Z}\sum_{m=0}^\infty \frac{e^{-2t} t^{2m+\mv-\mw}}{m!(m+\mv-\mw)!}f(\mw),\qquad t\ge 0,\ \mv\in \Z
\end{equation}
where the first representation formula is taken from~\cite[Exa.~12.3.3]{Dav07} (or, slightly more explicitly,~\cite[Thm.~5.2]{EstHamHat17}) and the second from~\cite[Prop.~2]{CiaGilRon14}.

Because $\mathcal A=2-\mathcal L$, \eqref{eq:estliz} immediately yields analogous representations
\begin{equation}\label{eq:estliz-adj}
e^{t\mathcal A(\Z)}f(\mv)=\sum_{\mw\in \Z}f(\mw) \frac{1}{2\pi}\int_{-\pi}^\pi \cos((\mv-\mw)q) e^{2t\cos q}dq=
\sum_{\mw\in \Z}\sum_{m=0}^\infty \frac{ t^{2m+\mv-\mw}}{m!(m+\mv-\mw)!}f(\mw),\qquad t\ge 0,\ \mv\in \Z\ .
\end{equation}
 for the semigroup that governs the \textit{forward} evolution equation for~\eqref{eq:mainpde}.
These formulae show that the heat kernel associated with $\mathcal A$ is for each $t\ge 0$ and all $\mv,\mw\in \mV$ a strictly positive number, hence the forward semigroup $(e^{t\mathcal A})_{t\ge 0}$ is a positive semigroup with infinite speed of propagation: this is not surprising, since $\mG$ is connected and hence $\mathcal A$ is irreducible. We conclude that the \textit{backward} semigroup $(e^{t\mathcal A})_{t\le 0}$ is irreducible although non-positive, hence it enjoys infinite speed of propagation as well. 
Of course, $\Z$ is in particular a tree all of whose vertices have degree 2. An extension of the above formulae to the case of trees all of whose vertices have degree $k$ can be deduced from corresponding formulae obtained in~\cite[\S~5]{ChuYau98}.

2) Further formulae are known. It is known how the adjacency matrix of the Cartesian product $\mG_1\square \mG_2$ of graphs $\mG_1,\mG_2$ can be written as
\[
{\mathcal A}(\mG_1\square \mG_2)=\mathcal A(\mG_1)\otimes I_{\mG_2}+I_{\mG_1}\otimes {\mathcal A}(\mG_2)
\]
see~\cite[Thm.~2.21]{CveDooSac79}. In particular, this allows to write the adjacency matrix of $\mathbb Z^n$ as a sum: indeed,
\begin{equation}\label{eq:adjz}
{\mathcal A}(\Z^n)=\mathcal A(\Z)\otimes I_{\Z}\otimes \cdots \otimes I_{\Z}+  I_{\Z}\otimes\mathcal A(\Z)\otimes \cdots \otimes I_{\Z} +I_{\Z}\otimes I_{\Z}\otimes\cdots\otimes 
\mathcal A(\Z)\ 
\end{equation}
and it generates the tensor product semigroup
\[
e^{t\mathcal A(\Z^n)}=\bigotimes_{k=1}^n e^{t\mathcal A(\Z)}\ ,
\]
cf.~\cite[\S~A.1 3.7]{Nag86} and~\cite[Chapter 4]{HorJoh91}. Since by the elementary properties of the Kronecker product the addends in~\eqref{eq:adjz} commute,  one may even obtain a more explicit formula as an application of the Lie--Trotter product formula based on~\eqref{eq:estliz-adj}. Alternatively, one can use $\mathcal A=2n-\mathcal L$ and hence 
\[
e^{t\mathcal A(\Z^n)}f=e^{2tn}e^{-t\mathcal L(\Z^n)}f=e^{2tn}(k_t\ast f)\ ,
\]
with the heat kernel $k_t$ simply being the product of $n$ copies of the $1$-dimensional heat kernels obtained in 1), cf.~\cite[Exa.~12.5.9]{Dav07}.

3) Observe that if we consider a new graph $\mG'$ obtained adding or deleting a finite number of edges to $\Z^n$, its adjacency matrix is a finite rank perturbation of $\mathcal A(\Z^n)$ and hence the two semigroups have the same essential growth bound, cf.~\cite[Prop.~IV.2.12]{EngNag00}.
\end{exa}

Finding analogous formulae for generic line graphs is out of question:
we are therefore rather inclined to study qualitative properties of the backward equation~\eqref{eq:mainpde} by abstract semigroup theory, exploiting the quadratic form associated with $\mathcal A$ in the case of line graphs.

\section{Forward evolution equation on uniformly locally finite graphs}\label{sec:ulfgraphs}

By Lemma~\ref{lem:mohar}, the forward Cauchy problem is
well-posed if the edge weights of $\mG$ are bounded. This sufficient condition is by Proposition~\ref{prop:mohar-reloaded} also necessary even in the rather special case of a line graph. In this section we are going to study some properties of the semigroup that governs the forward Cauchy problem
\begin{equation}\label{eq:mainpde-forw-C}
\begin{cases}
\frac{du}{dt}(t,\mv)&=\mathcal A_C u(t,\mv), \qquad t\ge 0,\ \mv\in\mV\ ,\\
u(0,\mv)&=u_0(\mv),\qquad \mv\in \mV\ ,
\end{cases}
\end{equation}
where $\mathcal A_C$ is defined in~\eqref{def:acnonline},
 under the standing assumption that 
\begin{itemize}
\item $\mG=(\mV,\mE)$ is a possibly infinite but locally finite graph \item whose edges are weighted by $0<c\in \ell^\infty(\mE)$ and \item such that $\deg_\mG^C\in \ell^\infty(\mV)$ (recall~\eqref{eq:def-degc}).
\end{itemize}
This means that $\mG$ is uniformly locally finite with respect to $c$. Of course, this includes finite graphs as special cases. We already know that $\mathcal A$ is a bounded operator and that the exponential series
\begin{equation}\label{eq:exponential}
e^{z\mathcal A_C}:=\sum_{k=0}^\infty \frac{z^k}{k!}\mathcal A_C^k\ , \qquad z\in \mathbb C\ ,
\end{equation}
is norm convergent for all $z\in \C$ and hence defines a bounded linear operator on $\ell^p(\mV)$ for all $p\in [1,\infty]$.

\begin{prop}\label{prop:positforward}
The following assertions hold.
\begin{enumerate}
\item The semigroup $(e^{t\mathcal A_C})_{t\ge 0}$ is positive.
\item  It is irreducible and hence positivity improving if and only if $\mG$ is connected.
\item It is not  contractive with respect to the norm of $\ell^\infty(\mV)$, but  the rescaled semigroup $(e^{-t\|\deg_{\mG}^C\|_\infty}e^{t\mathcal A_C})_{t\ge 0}$ is.
\end{enumerate}
\end{prop}

\begin{proof}
1) Since $\mathcal A_C$ is a bounded operator on $\ell^2(\mV)$ and hence $(e^{t\mathcal A_C})_{t\ge 0}$ is given by means of~\eqref{eq:exponential}, it is a positivity preserving semigroup since its generator $\mathcal A_C$ is positivity preserving.

2) If $\mG$ is disconnected, then the semigroup is clearly reducible. Conversely, let $\emptyset\neq \tilde{\mV}$ be an arbitrary proper subset of $\mV$ and consider the characteristic function ${\mathbf 1}_{\tilde{\mV}}$: we have to show that $e^{t\mathcal A}{\mathbf 1}_{\tilde{\mV}}(\mw)>0$ for at least one $\mw\in \mV\setminus\tilde{\mV}$. Now, this is clear for all finite induced subgraphs $\mG_n=(\mV_n,\mE_n)$ of $\mG$ that contain $\tilde{\mV}$ (since finite adjacency matrices are irreducible if and only if their graph is connected), and hence the assertion follows by a convergence result due to Mohar~\cite[Prop.~4.2]{Moh82}, which clearly extends to general weighted adjacency matrices.

3) In order for $(e^{t\mathcal A_C})_{t\ge 0}$ to be $\ell^\infty$-contractive, again in view of~\cite[Prop.~4.2]{Moh82} all its finite submatrices ought to generate $\ell^\infty$-contractive semigroups; but this is impossible, since each of them violates the condition in~\cite[Lemma~6.1]{Mug07}.

On the other hand, 
\[
e^{t(\mathcal A-\|\deg_{\mG}^C\|_\infty)}=e^{-t\|\deg_\mG^C\|}e^{t\mathcal A}=e^{t(\mathcal D-\|\deg_\mG^C\|_\infty)}e^{-t\mathcal L} ,\qquad t\ge 0,
\]
since 
\[-\mathcal L:=\mathcal A-\mathcal D
\] with $\mathcal A,\mathcal L,\mathcal D$ pairwise commuting (here we denote by $\mathcal L,\mathcal D$ as usual the Laplacian matrix and the diagonal matrix of vertex degrees, respectively). Accordingly, for all $u\in \ell^p(\mV)$ with $\|u\|_\infty \le 1$ we find
\[
\|e^{t(\mathcal A-\|\deg_{\mG}^C\|_\infty)}u\|_\infty
\le \|e^{t(\mathcal D-\|\deg_\mG^C\|_\infty)}\| \|e^{-t\mathcal L} u\|_\infty\le \|e^{-t\mathcal L} u\|_\infty\le 1 ,\qquad t\ge 0,
\]
owing to the fact that $\mathcal D-\|\deg_\mG^C\|_\infty$ is a diagonal matrix with nonpositive entries and that $(e^{-t\mathcal L})_{t\ge 0}$ is well-known to be associated with a Dirichlet form~\cite{KelLen12} and hence to be $\ell^\infty(\mV)$-contractive.
\end{proof}

We can finally apply a recent result on kernel semigroups in order to discuss long-time behavior of the heat kernel. We emphasize that it does not rely upon compactness assumptions -- indeed, $(e^{t\mathcal A_C})_{t\ge 0}$ can never be a compact semigroup if $\mG$ is infinite, since $\mathcal A_C$ is either bounded or not a generator.

\begin{prop}\label{prop:asymptforward}
For all $\mv,\mw\in \mV$
\begin{equation}\label{eq:kellenvog}
e^{-ts(\mathcal A_C)}(e^{t\mathcal A_C})_{\mv,\mw}\to \Phi(\mv)\Phi(\mw)\quad \hbox{as }t\to\infty\ ,
\end{equation}
where $\Phi:\mV\to \R$ is the normalized Perron--Frobenius eigenvector if the spectral radius $s(\mathcal A_C)$ is an eigenvalue of $\mathcal A_C$ and $\Phi\equiv 0$ otherwise.
Furthermore, 
\[
\lim_{t\to\infty} \frac{\log (e^{t\mathcal A_C})_{\mv,\mw}}{t}=s(\mathcal A_C)\qquad \hbox{for all }\mv,\mw\in \mV\ .
\]
\end{prop}

\begin{proof}
Observe that $s(\mathcal A_C)\le\|\deg_{\mG}^C\|_\infty <\infty$ by Proposition~\ref{prop:basic-finite}. The assertion thus follows from~\cite[Thm.~3.1]{KelLenVog15}.
\end{proof}

\begin{exa}
Let $\mG=\Z$.
It is known (see e.g.~\cite{KilSim03}) that the unweighted adjacency matrix $\mathcal A(\Z)$ has no eigenvalues and its spectrum is $[-2,2]$. Hence, by Proposition~\ref{prop:asymptforward} we deduce that for all $\mv,\mw\in \mV$
\[
e^{-2t}(e^{t\mathcal A})_{\mv,\mw}\to 0\quad \hbox{as }t\to\infty\ .
\]
\end{exa}

Comparison principles for the forward evolution equation
\begin{equation}\label{eq:mainpde-lapl-forw}
\frac{du}{dt}(t)=-\mathcal Lu(t)\ ,\qquad t\ge 0\ , 
\end{equation}
 associated with $-\mathcal L$ on different underlying graphs $\mG,\tilde{\mG}$ follow immediately from classical results of Fiedler and later authors. For instance, provided $\tilde{\mG},\mG$ share the same vertex set, convergence to equilibrium  for the solution of~\eqref{eq:mainpde-lapl-forw} is by~\cite[Cor.~3.2]{Fie73} at most as fast on $\tilde{\mG}$ as on $\mG$ if the edge set of $\tilde{\mG}$ is contained in the edge set of $\mG$. In the case of the forward Cauchy problem~\eqref{eq:mainpde-forw} we are able to describe a different behavior, namely, domination of $(e^{t\mathcal A(\tilde{\mG})})_{t\ge 0}$  by
 $(e^{t\mathcal A(\mG)})_{t\ge 0}$, i.e., the property that
\[
e^{t\mathcal A(\tilde{\mG})}f(\mv)\le e^{t\mathcal A({\mG})}f(\mv)\qquad\hbox{for all } t\ge 0,\ 0\le f\in \R^{|\mV|},\hbox{ and }\mv\in \mV\ .
\] 
The following is a direct consequence of~\cite[Thm.~2.24]{Ouh05}. We emphasize that it holds for general graphs (not necessarily line graphs!).
 
 \begin{prop}\label{prop:domindortm}
 Let $\mG,\tilde{\mG}$ two different graphs sharing the same vertex set $\mV$. Then the following assertions are equivalent.
 \begin{enumerate}[(i)]
\item  $(e^{t\mathcal A(\tilde{\mG})})_{t\ge 0}$ is dominated by
 $(e^{t\mathcal A(\mG)})_{t\ge 0}$.
 \item $\mathcal A(\tilde{\mG})_{\mv\mw}\le \mathcal A(\mG)_{\mv\mw}$ for all vertices $\mv,\mw\in \mV$.
 \item $\tilde{\mG}$ is a subgraph of $\mG$.
\end{enumerate} 
 \end{prop}

\begin{rem}
It is easy to see that Proposition~\ref{prop:domindortm} extends to perturbations of the adjacency matrix by diagonal matrices. In particular, for the Laplacian $\mathcal L$ one deduces that  $(e^{-t\mathcal L(\tilde{\mG})})_{t\ge 0}$ is never dominated by $(e^{-t\mathcal L(\mG)})_{t\ge 0}$ (as usual, in the sense of positivity preserving operators) if $\tilde{\mG}$ is a strict subgraph of $\mG$. However,  for the signless Laplacian $\mathcal Q$ one does find that  $(e^{t\mathcal Q(\tilde{\mG})})_{t\ge 0}$ is dominated by $(e^{t\mathcal Q(\mG)})_{t\ge 0}$.
\end{rem}

\begin{rem}
There are a few models based on the linear dynamical system~\eqref{eq:mainpde} in applied sciences: let us mention the possibly simplest model for the spread of diseases with recovery and repeated infection, the so-called \textit{SIS model}.
In a network context it can be formulated as
\begin{equation}\label{eq:SI}
\frac{du}{dt}(t)=\beta \mathcal Au(t)-\beta \diag(u(t))\mathcal Au(t)-\delta u(t)  ,\qquad t\ge 0\ , 
\end{equation}
for some $\beta,\delta>0$, see~\cite{WanChaWan03}: hence, the \textit{forward} evolution equation associated with~\eqref{eq:mainpde} can be seen as a simple linearisation of~\eqref{eq:SI}, up to scalar (additive and multiplicative) perturbations: more precisely, the solution of the linear part of~\eqref{eq:SI} for initial value $u(0)=u_0$ is given by
\[
u(t)=e^{-\delta t}e^{\beta t\mathcal A}u_0,\qquad t\ge 0\ ;
\]
therefore, investigating qualitative properties of the solution to~\eqref{eq:mainpde} helps unveiling features of the SIS model: well-posedness and regularity properties of~\eqref{eq:SI} on an \textit{infinite} network could then be obtained by techniques based on maximal regularity properties, cf.~\cite[Chapt.~7]{Lun95} or~\cite[Chapt.~10]{ChiFas10}. Unsurprisingly, it has been observed in the computer science community that the spectral radius of $\mathcal A$ is the main parameter when it comes to describe the long-time behavior of~\eqref{eq:SI}: indeed, an \textit{epidemic threshold} is determined by $\lambda_{\max}(\mathcal A)$, $\beta$, and $\delta$ alone, cf.~\cite{GanMasTow05,ChaWanWan08,vanOmiKoo09,van11b}.
\end{rem}

\begin{rem}
A different but related model is discussed in~\cite{SarBoy96}: for the purposes of segmentation of an image encoded by a vector $u_0\in \R^N$ (where $N$ is the number of pixels), the authors propose to study the Perron eigenvector of $\mathcal A$. As we know from~\eqref{eq:engnag2} that ~\eqref{eq:mainpde} converges to the Perron eigenprojector upon suitable rescaling, studying the semigroup generated by $\mathcal A$ can be regarded as a relaxation of the classic spectral approach: While the Perron eigenvector is obtained by taking $\lim_{t\to \infty}e^{-t\lambda_{\max}}e^{t\mathcal A}u_0$, depending on the relevant picture significant information might  be obtained also by means of $e^{t\mathcal A}u_0$ for finite $t$. Following an intuition by Chung~\cite{Chu07,Gle15}, $(e^{t\mathcal A}u_0)	_{t\ge 0}$ might then be regarded as a whole family of approximate segmentations.
\end{rem}

\begin{rem}
It was observed in~\cite[Prop.~2.2]{BalGolJer15} that the only graphs isomorphic to their own line graph are $\Z$, $\N$ and the cycle graphs, so the $\inf$ of the spectrum of $\mathcal A$ on these graphs is not smaller than $-2$.  On the other hand, the line graph of $\Z^n$ is a $(4n-2)$-regular graph, so $\mathcal A$ is a bounded and hence self-adjoint operator. However, the spectrum of the adjacency matrix on $\Z^n$ is $[-2n,2n]$ (see e.g.~\cite[\S~7]{MohWoe89}, so $\mathbb Z^n$ is not a line graph for any $n\ne 1$.

In the case of $\mG=\mH=\mathbb Z$, the degree function is constant and hence $\mathcal A$ is self-adjoint and it is known that its essential spectrum is $[-2,2]$.
The spectral theory of $\mathcal A$ has been thoroughly studied in~\cite{KilSim03} as a special case of a Jacobi matrix -- i.e., of a discrete Schrödinger operator $-\mathcal L-V$ on $\Z$.
 In particular, a characterization of the Hilbert--Schmidt property of $\mathcal A-(-\mathcal L-V)$ in terms of $V$ has been obtained in~\cite[Thm.~1]{KilSim03}. A necessary condition for the absolutely continuous spectrum of $\mathcal L-V$ to be $[-2,2]$ has been found in~\cite[Thm.~7]{KilSim03}; finally, \cite[Thm.~8]{KilSim03} states that $\mathcal A$ is the only bounded perturbation of $-\mathcal L$ with empty point spectrum. Besides the intrinsic interest in spectral theory, these results by Killip and Simon are also relevant for our parabolic setting: e.g., \cite[Prop.~IV.2.12]{EngNag00} states that if  $\mathcal A-(-\mathcal L-V)$ is compact, then the semigroups generated by $\mathcal A$ and by $-\mathcal L-V$ have same essential growth bound.
 
 If however the graph is not uniformly locally finite, then $\mathcal D$ is not a bounded perturbation -- in fact, at least on sparse graphs, not even a small form perturbation by~\cite[Thm.~1.1]{BonGolKel15} -- and one cannot generally conclude that $\mathcal A$ generates a semigroup if so does $-\mathcal L$. 
Thus, what our results show is that there is a class of graphs for which certain discrete Schrödinger operators enjoy forward well-posedness even if the free Laplacian does not. Conversely, our results also signify that Kato-type theory 
for discrete Laplacians is trickier than in the classical space-continuous setting: either $\mathcal D$ is bounded, or $-\mathcal L+\mathcal D$ is self-adjoint but not bounded from above.
\end{rem}

Let us conclude this section by briefly discussing the interplay of the dynamical system~\eqref{eq:mainpde} with the symmetries of the underlying graph $\mG$.

A \emph{permutation} is a symmetric matrix whose entries are all $0$ apart from exactly one $1$ on each of its columns/rows. An \emph{automorphism} is
a permutation that commutes with $\mathcal A(\mG)$.

Each permutation is clearly a doubly stochastic matrix. Tinhofer introduced in~\cite{Tin86} the notion of 
\emph{doubly stochastic automorphism} of a graph $\mG$
 as a matrix that commutes with $\mathcal A(\mG)$ but is merely doubly stochastic. 
 Automorphisms are of course double stochastic automorphisms, and so are averaging operators over orbits of automorphisms and more generally orthogonal projectors associated with equitable partitions~\cite{GodMcK80}; however, there exist graphs that admit doubly stochastic automorphisms which are \textit{not} convex combinations of automorphisms -- among other, the Petersen graph~\cite{EvdKarPon99}. Tinhofer's definition can be extended to infinite graphs verbatim.
 
 \begin{defi}
 An infinite matrix $\mathcal O$ is called \emph{doubly stochastic} if each of its entries is nonnegative and each of its rows as well as each of its columns has entries summing up to 1.
 \end{defi}

 \begin{lemma}
 Each doubly stochastic automorphism indexed in $\mV\times \mV$ defines a bounded linear operator -- in fact, a contraction -- on $\ell^p(\mV)$ for all $p\in [1,\infty]$.
 \end{lemma}

\begin{proof}
 The assertion for $p=2$ holds as a direct consequence of~\cite[Thm.~6.12-A]{Tay58}, whereas contractivity can be checked directly for $p=\infty$. The general case is proved by interpolation and duality.
\end{proof}

\begin{defi}
A \textit{doubly stochastic automorphism} of $\mG$ is a doubly stochastic matrix $\mathcal O$ indexed in $\mV\times \mV$ such that
$\mathcal{OA}=\mathcal {AO}$.
\end{defi}

Observe that this commutation relation is well-defined because both $\mathcal A$ and $\mathcal O$ are bounded operators. We conclude that doubly stochastic automorphisms are symmetries for the parabolic equation~\eqref{eq:mainpde}, as long as $(e^{z\mathcal A})_{z\in \C}$ is defined by means of the exponential formula.

\begin{prop}
Let $\mG$ be uniformly locally finite and $\mathcal O$ be a doubly stochastic automorphism of $\mG$. Then 
\[
\mathcal O e^{z\mathcal A}=e^{z\mathcal A}\mathcal O\qquad\hbox{  for all }z\in \mathbb C\ .
\]
\end{prop}

\section{Backward evolution equation on general line graphs}

In Section~\ref{section:infinite} we have proved that 
$\mathcal A_C-\|\deg_{\mG}^C\|_\infty $ is associated with a Dirichlet form and thus generates a sub-Markovian semigroup. This is remarkable because $(e^{-t\|\deg_{\mG}^C\|_\infty }e^{t\mathcal A_C})_{t\ge 0}$ does not seem to have an obvious interpretation as a diffusion-like semigroup.

However, Proposition~\ref{prop:mohar-reloaded} precludes the possibility of well-posedness of the forward Cauchy problem~\eqref{eq:mainpde-forw-C} unless $\mG$ is uniformly locally finite with respect to the edge weight. It thus seems that for graphs of unbounded degree it is the backward Cauchy problem
\begin{equation}\label{eq:mainpde-backw-C}
\begin{cases}
\frac{du}{dt}(t,\mv)&=\mathcal A_Cu(t,\mv), \qquad t\le 0,\ \mv\in\mV\ ,\\
u(0,\mv)&=u_0(\mv),\qquad \mv\in \mV\ ,
\end{cases}
\end{equation}
 the more appropriate playground for the adjacency matrix.

Throughout this section we are going to impose the following assumptions:
\begin{itemize}
\item $
\mG=(\mV,\mE)$ is the line graph of a locally finite, connected graph $\mH=(\mV',\mE')$\footnote{ In particular $\mV=\mE'$, so we will not distinguish between $\mv\in \mV$ and the corresponding $\me'\in \mE'$.};
\item the edges of $\mH$ are weighted by $c:\mV'\to \R$ s.t.\ $\Gamma \ge c(\mv')\ge \gamma$  for some $\Gamma>\gamma>0$  and all $\mv'\in \mV'$;
\item the edges of $\mG$ are weighted accordingly to~\eqref{eq:adjweighted}.
\end{itemize}

\begin{lemma}
The vector space 
\[V:=\{u\in \ell^2(\mV): \mathcal J u\in \ell^2(\mV')\}
\]
is dense in $\ell^2(\mV)\simeq \ell^2(\mE')$. It is a Hilbert space with respect to the inner product
\[
\begin{split}
(u|v)_V&:=(\mathcal J u|\mathcal J v)_{\ell^2(\mV')}+(u|v)_{\ell^2(\mV)}\\
&=\sum_{\mv'\in \mV'}\left|\sum_{\me'\sim \mv'} u(\me') \right|^2+\sum_{\me'\in \mE'}\left|u(\me') \right|^2\ .
\end{split}
\]
\end{lemma}

\begin{proof}
The space $V$ is dense in $\ell^2(\mV)$ because so is $c_{00}(\mV)\subset V$. It is a Hilbert space because
\[
V\ni u\mapsto (u,\mathcal Ju)\in \ell^2(\mV)\times \ell^2(\mV')
\]
is an isometry between $V$ and $\ell^2(\mV)\times \ell^2(\mV')$. 
\end{proof}

We study again the quadratic form 
\[
a_C(u):=\sum_{\mv'\in\mV'} c(\mv')|\mathcal J u(\mv')|^2-\sum_{\mv\in\mV}\gamma(\mv)|u(\mv)|^2\ ,
\]
this time with domain
\[
D^N(a_C):=V
\]
or
\[
D^D(a_C):=\overline{c_{00}(\mV)}^{\|\cdot\|_V}\ .
\]

\begin{rem}
1) Line graphs of radial trees of unbounded degree are instances of graphs that can be treated by the techniques of this section but not those of Section~\ref{sec:ulfgraphs}.

2) If we add or delete a finite number of edges of a line graph $\mG$, the resulting graph $\tilde\mG$ will in general not be a line graph. However, the new adjacency matrix $\mathcal A(\tilde\mG)$ will be a finite rank perturbation of the operator associated with the quadratic form for $\mG$, hence all results of this section will apply to $\mathcal A(\tilde\mG)$ as well.
\end{rem}

It is immediate to check that the symmetric form $a_C$ is elliptic in $\ell^2(\mV)$ and bounded both with domain $D^N(a_C)$ and $D^D(a_C)$. Hence, the following holds.

\begin{lemma}
The associated operators
\begin{equation}\label{eq:defin-adj-neum}
\begin{split}
D(\mathcal A^N_C)&:=\{u\in D^N(a_C):\exists v\in \ell^2(\mV) \hbox{ s.t. }(v|w)_{\ell^2(\mV)}=a(u,w)\hbox{ for all }w\in D(a_C)\}\\
\mathcal A^N_C u&:=v
\end{split}
\end{equation}
and
\begin{equation}\label{eq:defin-adj-dir}
\begin{split}
D(\mathcal A^D_C)&:=\{u\in D^D(a_C):\exists v\in \ell^2(\mV) \hbox{ s.t. }(v|w)_{\ell^2(\mV)}=a(u,w)\hbox{ for all }w\in D_0(a_C)\}\\
\mathcal A^D_C u&:=v
\end{split}
\end{equation}
are self-adjoint and both $-\mathcal A^N-2\Id , -\mathcal A^D-2\Id $ are dissipative.
\end{lemma}

We are unable to provide combinatorial conditions under which $\mathcal A^N$ and the Friedrichs extension $\mathcal A^D$ agree, but it has been recently observed in~\cite{BalGolJer15} that line graphs of certain growing \emph{bipartite} graphs -- including so-called \textit{anti-trees} (but not trees!) -- do have essentially self-adjoint adjacency matrices.

Observe that in the case of graphs that are not uniformly locally finite the operator $\mathcal A$ is \textit{defined variationally} by means of the quadratic form $a$, whereas~\eqref{Ac00} is merely a property of $\mathcal A_C$. 

Thus, the general theory of operator semigroups associated with quadratic forms (see \cite[Chapter~6]{Mug14} for a compact overview) yields the following.

\begin{theo}
The operator $-\mathcal A_C^D$ (resp., $-\mathcal A_C^N$) associated with $a_C$ with domain $D_0(a_C)$ (resp., $D(a_C)$) generates a quasi-contractive, analytic semigroup of angle $\frac{\pi}{2}$  as well as a cosine operator function on $\ell^2(\mV)$. Accordingly, the Cauchy problems for both the backward parabolic equation
\begin{equation}\label{eq:mainpde-aheat}
\begin{cases}
\frac{du}{dt}(t,\mv)&=\mathcal A^*_Cu(t,\mv)\ ,\qquad t\le 0,\ \mv\in\mV\ , \\
u(0,\mv)&=u_0(\mv),\qquad \mv\in \mV\ ,
\end{cases}
\end{equation}
(for $u_0\in \ell^2(\mV)$)
and the hyperbolic equation
\begin{equation}\label{eq:mainpde-awave}
\begin{cases}
\frac{d^2u}{dt^2}(t)&=\mathcal A^*_Cu(t)\ ,\qquad t\in \mathbb R\ ,\\
u(0,\mv)&=u_0(\mv),\qquad \mv\in \mV\ ,\\
\frac{du}{dt}(0,\mv)&=u_1(\mv),\qquad \mv\in \mV\ ,
\end{cases}
 \end{equation}
(for $u_0\in \ell^2(\mV)$ and $u_1\in D^*(a_C)$)
are well-posed for both $*=D$ and $*=N$.
\end{theo}

\begin{cor}
Given a sequence $(m_\mv)_{\mv\in \mV}$ and considering the potential
\[
(Vu)_\mv:=m_\mv u_\mv,\qquad \mv\in \mV,
\]
the backward parabolic equation
\begin{equation}\label{eq:mainpde-laplheat}
\frac{du}{dt}(t,\mv)=-\mathcal Lu(t,\mv)+Vu(t,\mv)\ ,\qquad t\le 0\ , \mv\in \mV\ ,
\end{equation}
is well-posed, provided there is $M>0$ such that 
\begin{equation}\label{eq:cond-bddperturb}
|m_\mv-\deg(\mv)|\le M\quad \hbox{for all }\mv\in \mV.
\end{equation}
\end{cor}

\begin{proof}
Since under~\eqref{eq:cond-bddperturb} $m-\deg\in \ell^\infty(\mV)$, one sees that the Schrödinger operator $-\mathcal L+V$ is a bounded perturbation of $-\mathcal L+\mathcal D=\mathcal A$, hence classical generation results for perturbed operators apply.
\end{proof}

\begin{prop}
The backward semigroup $(e^{t\mathcal A^*})_{t\le 0}$ is not positive for either $*=D$ or $*=N$ if $\mG$ has at least two vertices, or equivalently if $\mH$ has at least two edges.
\end{prop}

\begin{proof}
We want to apply the characterization of positivity preserving semigroups associated with quadratic forms in~\cite[Thm.~2.7]{Ouh05} to the form $a_C$, but in fact up to a scalar perturbation (that does not affect positivity of the generated semigroup) we may equivalently study the quadratic form
\[
\tilde{a}(u)=\sum_{\mv'\in\mV'} c(\mv')|\mathcal J u(\mv')|^2=\sum_{\mv'\in \mV'}c(\mv')\left|\sum_{\me'\sim \mv'} u(\me')\right|^2,\qquad u\in D^*(a_C).
\]
Accordingly, we need to check that
\begin{equation}\label{eq:ouhab-equiv}
\exists u\in D^*(a_C) \quad \hbox{ s.t. }\quad u^+\not\in D^*(a_C) \hbox{ or } \tilde{a}(u^+)> \tilde{a}(u).
\end{equation}
Let us pick $\mv_0'\in \mV'$ with $\deg_\mH(\mv_0')\ge 2$ and take two edges $\me'_0,\me'_1\in \mE'$ that are incident in $\mv'_0$. Now,
\[
\left|\sum_{\me'\sim \mv'_0} u(\me')\right|< \left|\sum_{\me'\sim \mv'_0} u^+(\me')\right|
\]
for the function $u$ defined by
\[
u(\me'):=\begin{cases}
1 &\hbox{if }\me'=\me'_0\\
-1\quad &\hbox{if }\me'=\me'_1\\
0 &\hbox{otherwise}.
\end{cases}
\]
This yields~\eqref{eq:ouhab-equiv} and concludes the proof.
\end{proof}

\begin{rem}
Since $\mathcal A$ is a real symmetric matrix, we know from~\cite[Lemma~1.2.8]{Dav96} that $\mathcal A$ has at least one self-adjoint extension. Does it have \emph{exactly} one self-adjoint extension, i.e., is $\mathcal A$ essentially self-adjoint (in particular: do $\mathcal A^D$ and $\mathcal A^N	$ agree)? In order to answer this question, Golénia has introduced in~\cite{Gol10} the condition that the degree function, although not necessarily bounded, has \textit{bounded variation}, i.e.,
\begin{equation}\label{eq:condgolenia}
\sup\limits_{\mv\in \mV}\max_{\mw\sim \mv}|\deg(\mv)-\deg(\mw)|<\infty\ :
\end{equation}
this is e.g.\ the case for rooted radial trees each of whose vertices at distance $n$ from the root has at most $n+k$ offsprings, for a certain fixed $k\in \mathbb N$; or for their line graphs. 
It is proved in~\cite[Prop.~1.1]{Gol10} by means of Nelson's commutator theorem that unweighted graphs satisfying~\eqref{eq:condgolenia} have essentially self-adjoint adjacency matrices.

Stars are the simplest example of graph with degree function of high variation; but stars (on at least four vertices) cannot be line graphs in view of Beineke's Theorem, so one may wonder whether line graphs automatically satisfy Golénia's condition~\eqref{eq:condgolenia}. 
However, it has been shown in~\cite[Cor.~5.3]{BalGolJer15} that even adjacency matrices that are bounded from below, and in particular adjacency matrices of line graphs, can fail to be essentially self-adjoint.

Let $\tilde{A}$ be a self-adjoint extension of $\mathcal A_0:=\mathcal A_{|c_{00}(\mV)}$.
By the Lumer--Phillips Theorem (or simply by the Spectral Theorem), $\tilde{\mathcal A}$ (resp., $-\tilde{\mathcal A}$) generates a semigroup on $\ell^2(\mV)$ if and only if $\tilde{\mathcal A}$ is bounded from above (resp., from below). The following assertions have been shown by Golénia.
\begin{itemize}
\item Let the edge weight function be bounded from below away from 0; then any self-adjoint extension of $\mathcal A_0$ is bounded if and only if both the (unweighted) degree function and the edge weight function are bounded~\cite[Prop.~3.1]{Gol10}.
\item Let the edge weight function be unbounded; then any self-adjoint extension of $\mathcal A_0$ is unbounded both from above and from below~\cite[Prop.~3.2.(1)]{Gol10}.
\item Let the edge weight function be bounded from below away from 0; then boundedness from above of any self-adjoint extension of $\mathcal A_0$ implies boundedness of the (unweighted) degree function~\cite[Prop.~3.2.(2)]{Gol10}.
\end{itemize}

Generation of a semigroup by $\pm A$ seems to be a stronger requirement on $\mG$. Take for example $\mG$ to be the line graph of a radial tree with unbounded degree function: then the Cauchy problem for
\[
\frac{du}{dt}(t,\mv)=(\mathcal A-\mathcal D)u(t,\mv),\qquad t\ge 0\ ,\mv\in \mV\ ,
\]
is well-posed, but the Cauchy problem for
\[
\frac{du}{dt}(t,\mv)=\mathcal A u(t,\mv),\qquad t\ge 0\ ,\mv\in \mV\ ,
\]
is not. Indeed, on any unweighted graphs $\mG$ the Laplacian $-\mathcal L_{|c_{00}(\mV)}$ is a negative semidefinite and essentially self-adjoint operator (see e.g.~\cite[Example 1]{HaeKelLen12}), hence its closure generates a semigroup on $\ell^2(\mV)$. On the other hand, by~\cite[Prop.~3.1]{Gol10} every self-adjoint extension of $\mathcal A_{|c_{00}(\mV)}$ on a line graph is bounded from above if and only if the degree function is bounded.
\end{rem}

Our discussion of the long time asymptotics of $(e^{\pm t\mathcal A})_{t\ge 0}$ was based on the general properties of (finite) Hermitian matrices in the case of finite $\mG$, while we could discuss the limit of $(e^{ t\mathcal A})_{t\ge 0}$ based on Perron--Frobenius theory for generic uniformly locally finite graphs. Both approaches fail whenever $(e^{ t\mathcal A})_{t\ge 0}$ is studied, but we can still invoke compactness arguments.

In view of Pitt's theorem, a sufficient condition for the backward semigroup $(e^{t\mathcal A})_{t\le 0}$ to be compact is that it maps $\ell^2(\mV)$ into $\ell^p(\mV)$ for any $p\in [1,2)$. Unfortunately, checking this condition requires good knowledge of the semigroup kernel. However, it is known that a semigroup associated with a form is compact if and only if the form domain is compactly embedded in the ambient Hilbert space.

\begin{lemma}\label{lem:compact}
The embedding of $D^N(a_C)$ 
 into $\ell^2(\mV)$ is compact if for every $\epsilon>0$ there are $\mv\in \mV$ and $r>0$ such that
\begin{itemize}
\item $B(\mv,r)$ is finite and
\item there holds
\[
\sum_{\mw\not\in B(\mv,r)}|u(\mw)|^2<\epsilon^2
\]
for all $u\in \ell^2(\mV)$
 with $\|u\|^2_{\ell^2(\mV)}+\|\mathcal J u\|^2_{\ell^2(\mV')}\le 1$.
\end{itemize}
\end{lemma}

Here we have denoted by $B(\mv,r)$ ball of radius $r$ centered at $\mv$, i.e., the set of all vertices at distance at most $r$ from $\mv$.
Observe that compactness of this embedding implies in particular that the graph is not uniformly locally finite.

\begin{proof}
Compactness of the embedding into $\ell^2(\mV)$ under the above assumptions follows in a way similar to~\cite[Prop.~3.8]{Mug14}
by applying a result due to Hanche-Olsen and Holden, cf.~\cite[Thm.~4]{HanHol10}.
\end{proof}

Under the assumptions of Lemma~\ref{lem:compact} we deduce that $\mathcal A$ has pure point spectrum and the semigroup $(e^{-t\mathcal A})_{t\ge 0})$ is compact, hence by the Arendt--Batty--Lyubich--V\~{u} Theorem
\[
\lim_{t\to\infty}e^{t(\lambda_{\min}+\epsilon)}e^{-t\mathcal A}u=0\quad \hbox{for all }u\in \ell^2(\mV)
\]
for all $\epsilon>0$, cf.~\cite[Cor.~V.2.22]{EngNag00}, whereas $e^{t\lambda_{\min}}e^{-t\mathcal A}$ converges in operator norm towards the eigenprojector $P_{\min}$.
Still, in the infinite case it look less promising to apply these results due to lack of information on the spectrum of $\mathcal A$.

\section{An interpretation of the linear dynamical system associated with $\mathcal A$}

We have seen throughout this paper that both $(e^{t\mathcal A})_{t\le 0}$ and especially -- whenever it exists -- $(e^{t\mathcal A})_{t\ge 0}$ display nice analytic properties that resemble those of more common models of diffusion. In certain cases we have also been able to identify the equilibria towards which the backward and/or forward evolution equation is converging. However, it is not quite obvious what behavior is typical for solutions of~\eqref{eq:mainpde}, and hence what kind of evolutionary systems can be modeled by means of the adjacency matrix $\mathcal A$.

\begin{exa} In the case of small finite graphs, the exponential matrix generated by $\mathcal A$ can be computed explicitly.
\begin{itemize}
\item Let $\mG=\mP_2$, the graph consisting of two vertices and one edge only. Then, 
\[
e^{t\mathcal A(\mP_2)}= 
\begin{pmatrix} \cosh(t) & \sinh(t)\\
 \sinh(t) & \cosh(t)
 \end{pmatrix}\ ,
\]
For comparison, the group generated by (minus) the discrete Laplacian is given by the known formula
\[
e^{-t\mathcal L(\mP_2)}= \frac{1}{2}
\begin{pmatrix} 1+e^{-2t} & 1-e^{-2t}\\
1-e^{-2t} & 1+e^{-2t}
 \end{pmatrix}\ .
\]
This can also be seen from~\eqref{eq:explicit-regular}.
\item It is more interesting to consider a simple case that is not covered by~\eqref{eq:explicit-regular}.
If $\mG=\mP_3$, a path on three vertices, then
\[
e^{t\mathcal A(\mP_3)}= \frac{1}{4}
\begin{pmatrix} 2+\cosh(\sqrt{2} t) & \sqrt{2}\sinh(\sqrt{2}t) & -2+\cosh(\sqrt{2} t)\\
\sqrt{2}\sinh(\sqrt{2}t)& \cosh(\sqrt{2} t)& \sqrt{2}\sinh(\sqrt{2}t)\\
-2+\cosh(\sqrt{2} t) & \sqrt{2}\sinh(\sqrt{2}t)&2+\cosh(\sqrt{2} t)
 \end{pmatrix}\ ,
\]
For comparison, the group generated by (minus) the discrete Laplacian is given by 
\[
e^{-t\mathcal L(\mP_3)}= \frac{1}{3}
\begin{pmatrix} 
1+\frac{3e^{-t}+e^{-3t}}{2} & 1-e^{-3t} & 
1+\frac{e^{-3t}-3e^{-t}}{2}\\
1-e^{-3t}& 1+2e^{-3t}& 1-e^{-3t}\\
1+\frac{e^{-3t}-3e^{-t}}{2} &  1-e^{-3t} & 1+\frac{3e^{-t}+e^{-3t}}{2}
 \end{pmatrix}\ ,
\]

\item Let $\mG=\mC_4$, the cycle graph on $4$ edges. Then 
\begin{equation*}\label{eq:matrixexpoC4}
e^{t\mathcal A(\mC_4)}= \frac12
\begin{pmatrix}
1 +\cosh(2t) &
\sinh(2t) &1+\cosh(2t) & \sinh(2t)\\
\sinh(2t) & 1 +\cosh(2t) & \sinh(2t) & 1+\cosh(2t)\\
1+\cosh(2t) & \sinh(2t) & 1+\cosh(2t) & \sinh(2t)\\  
\sinh(2t) & 1+\cosh(2t) & \sinh(2t) & 1+\cosh(2t)
\end{pmatrix}\ .
\end{equation*}
\end{itemize}
\end{exa}

Let us mention the following immediate consequence of~\cite[Thm.~3.3]{Zag17}.

\begin{prop}
Let $H$ be a Hilbert space and $C$ be a self-adjoint and bounded linear operator on $H$. Then 
there exists $M>0$ such that
\[
\left\|\frac{C^n}{\|C\|^n}-\frac{e^{\frac{n}{\|C\|}C}}{e^n} \right\| \le \frac{M}{\sqrt[3]{n}}\qquad \hbox{for all }n\in \mathbb N\ .
\]
\end{prop}
 
In the case of $C=\mathcal A$, the adjacency matrix of a graph, this is interesting because the powers $\mathcal A^n$ have a well-known combinatorial interpretation: indeed, for each $n$ the $\mv\mw$-entry of $\mathcal A^n$ is precisely the number of walks of length $n$ from $\mv$ to $\mw$ (a walk being 	a sequence vertex-edge-vertex-$\ldots$-edge-vertex in which vertices can be repeated.)

\begin{cor}
Let $\mG$ be uniformly locally finite. Then there exists $M>0$ such that
\[
\left\|\frac{(\pm\mathcal A)^n}{\|\mathcal A\|^n}-\frac{e^{\pm\frac{n}{\|\mathcal A\|}\mathcal A}}{e^n} \right\| \le \frac{M}{\sqrt[3]{n}}\qquad \hbox{for all }n\in \mathbb N\ .
\]
\end{cor}

Additionally, we can describe the behavior of the dynamical system associated with $\mathcal A$ through its interplay with the dynamical systems associated with the discrete Laplacian $\mathcal L$ and the signless Laplacian $\mathcal Q$ as follows.

\begin{prop}
Let $\mG$ be uniformly locally finite. Then
\[
0\le e^{-t\|\deg_{\mG}^C\|_\infty } e^{t\mathcal A}f \le e^{-t\mathcal L} f\le e^{t\mathcal A}f\le e^{t\mathcal Q}f\le e^{t\|\deg_{\mG}^C\|_\infty } e^{t\mathcal A}f\qquad \hbox{for all }0\le f\in \ell^2(\mV),\ t\ge 0\ .
\]
\end{prop}

\begin{proof}
Since $(e^{t\mathcal A})_{t\ge 0}$ is a positive semigroup, $e^{t\mathcal A}f(\mv)\ge 0$ for all $t\ge 0$ and all $\mv\in \mV$. Additionally, $\deg(\mv)>0$ for all $\mv\in\mV$, hence
 \[
e^{-t\|\deg_{\mG}^C\|_\infty } e^{t\mathcal A}f\le e^{-t\deg(\mv)}e^{t\mathcal A}f(\mv)\le e^{t\mathcal A}f(\mv)\le e^{t\deg (\mv)}e^{t\mathcal A}f(\mv)\le e^{t\|\deg_{\mG}^C\|_\infty } e^{t\mathcal A}f
\] for all $t\ge 0$ and all $\mv\in \mV$. Finally, observe that the diagonal matrix $\mathcal D=\diag(\deg(\mv))_{\mv\in \mV}$ commutes with $\mathcal A$, hence
\[
e^{t(\mathcal A\pm \mathcal D)}=e^{\pm t\mathcal D}e^{t\mathcal A},\qquad t\ge 0\ .
\]
This concludes the proof, since $\mathcal Q=\mathcal D+\mathcal A$ and $-\mathcal L=-\mathcal D+\mathcal A$ for uniformly locally finite graphs.
\end{proof}

We also remark that by~\cite[Thm.~C-II 4.17]{Nag86} $(e^{t\mathcal A})_{t\ge 0}$ is the modulus semigroup of $(e^{t\mathcal A})_{t\le 0}$, i.e., the domination property
\[
|e^{-t\mathcal A}f|\le e^{t\mathcal A}|f|\qquad \hbox{for all }f\in \ell^2(\mV),\ t\ge 0
\]
holds.

Finally, let us provide evidence that the forward problem for~\eqref{eq:mainpde} has, at least for short time, diffusive nature. 

\begin{prop}
Let $\mG$ be uniformly locally finite. Then
\[
\|e^{t\mathcal A}-e^{-t\mathcal L}\|\le t\|\deg_{\mG}^C\|_\infty e^{t\|\deg_{\mG}^C\|_\infty }\ \quad\hbox{for all }t\ge 0\ .
\]
\end{prop}

\begin{proof}
 Then $\mathcal D,\mathcal L$ are two commuting bounded operators and therefore
\[
\|e^{t\mathcal A}-e^{-t\mathcal L}\|=\|(e^{t\mathcal D}-\Id)e^{-t\mathcal L}\|\le \|(e^{t\mathcal D}-\Id)\|=e^{t\|\deg_{\mG}^C\|_\infty }-1\ 
\]
for all $t\ge 0$.
\end{proof}

\section{Miscellaneous comments}

\subsection{Non-self-adjoint evolution}

Let
\begin{itemize}
\item  $c:\mV'\to \R$ with $\Gamma,\gamma\in \R$ such that $\Gamma\ge c(\mv')\ge \gamma>0$ for all $\mv'\in\mV'$;
\item $B\in \mathcal L(\ell^2(\mV'),\ell^2(\mV))$;
\item $p\in \ell^\infty(\mV)$
\end{itemize}
Then we can consider the sesquilinear form $\mathfrak{a}$ defined by
\[
\mathfrak{a}(u,v):=(c\mathcal J u|\mathcal Jv)_{\ell^2(\mV')}+(B\mathcal Ju|v)_{\ell^2(\mV)}+(p u|v)_{\ell^2(\mV)}
\]
for $u,v$ in either of the form domains $D^N(a_C)$, $D^D(a_C)$ or (if $\mG$ is uniformly locally finite) $\ell^2(\mV)$. It is easy to see that $(a_C-\mathfrak{a})$ is a small form perturbation in the sense of~\cite[Lemma~2.1]{Mug08}, hence $\mathfrak{a}$ is a non-symmetric, elliptic, bounded form associated with an operator that acts on function with finite support as
\[
u\mapsto \mathcal J^Tc\mathcal Ju+B\mathcal Ju+pu\ ;
\]
these operators in turn generate analytic, quasi-contractive semigroups on $\ell^2(\mV)$.

\subsection{Quasilinear evolution}

In analogy with the theory of discrete $p$-Laplacians, see~\cite{Mug13} and references therein, one may extend the theory presented in this paper and consider the operator defined as the Fréchet derivative (with respect to $\ell^2(\mV)$) of the energy functional
\[
\Ffun_p(u):=\frac{\|\mathcal Ju\|^p_{\ell^p(\mV')}}{p}-2^{p-1}\frac{\|u\|^2_{\ell^2(\mV)}}{2}\ .
\]
Observe that the operator $\mathcal A_p$ associated with functional acts on functions $u$ with finite support as
\[
\mathcal A_p u=\mathcal J^T\left( |\mathcal J u|^{p-2} \mathcal Ju\right)-2^{p-1}u\ .
\]
In the prototypical case of even cycles $\mG$, $\mathcal A_p$ has minimal and maximal eigenvalues $\pm 2^{p-1}$, with the corresponding eigenfunctions being the same as for $p=2$ -- viz. constant functions for $2^{p-1}$ and alternating functions whose $\pm1$ patterns follows the bipartition for $-2^{p-1}$. The abstract theory of gradient systems in Hilbert space~\cite{ChiFas10} then directly yields that $-\mathcal A_p$ generates a semigroup of nonlinear operators on $\ell^2(\mV)$.

\subsection{Generalized line graphs}

Let $\mH=(\mV',\mE')$ be a graph and consider a vector in $\N^{|\mV|}$. The graph $\tilde{\mH}_{n_{\mv'_1},\ldots,n_{\mv'_{|\mV'|}}}$ is the multigraph obtained by adding to each $\mv'\in \mV'$ $n_{\mv}$ petals -- a petal being a pair of parallel edges incident in both $\mv$ and a new vertex $\mv'$.

\begin{defi}
A graph $\mG$ is called a \textit{generalized line graph}  if there exists a graph $\mH$ and a vector in $\N^{|\mV|}$ such that $\mG$ is the line graph  of $\tilde{\mH}=\tilde{\mH}_{n_{\mv_1},\ldots,n_{\mv_{|\mV'|}}}$. In this case, $\tilde{\mH}$ is called a \textit{root multigraph} of $\mG$.
\end{defi}

Clearly, line graphs are generalized line graphs, but the converse is not true. 

Historically, the main reason for considering generalized line graphs is that their adjacency matrix satisfies
\[
\mathcal A=\tilde{\mathcal J}^T \tilde{\mathcal J}-2\Id_\mV\ .
\]
Here $\tilde{\mathcal J}$ is an incidence-like matrix of $\tilde{\mH}$ defined just like classical incidence matrices $\mathcal J$ of simple graphs, with the following exception: if $\me',\mf'$ are parallel edges between an ``old'' vertex $\mv'\in \mV'$ and a ``new'' vertex $\mw'\in \tilde{\mV}'$, then the corresponding entries $\tilde{\mathcal J}_{\mw' \me'},\tilde{\mathcal J}_{\mw' \mf'}$ are set equal to $+1$ and $-1$, respectively (the precise choice is irrelevant, these entries only need to have different sign).

\begin{rem} Equivalently, $\mG$ is called a generalized line graph  if $\mathcal A(\mG)+2\Id_\mV$ is the Gramian matrix of a subset $D_\mG$ of
\[
\{\pm e_i\pm e_j\in \R^{|\mV|}:i\ne j\}
\]
where $e_k$ is the standard basis vector of $\R^{|\mV|}$, cf.~\cite[Chapt.~12]{GodRoy01}; hence, the adjacency matrix of a generalized line graph satisfies
\[
\mathcal A(\mG)+2\Id_\mV=U^T U
\]
where $U$ is the matrix whose columns are the elements of $D_\mG$.
\end{rem}

The following counterpart of Lemma~\ref{lemma:basic-spectral-finite} is known to hold form generalized line graph, cf.~\cite[Thm.~2.2.8]{CveRowSim04}.

\begin{lemma}
Let $\mG$ be a generalized line graph with root multigraph $\tilde{\mH}_{n_{\mv_1},\ldots,n_{\mv_{|\mV'|}}}$, then $-2$ is an eigenvalue of $\mathcal A(\mG)$ if and only if 
\[
|\mE'|-|\mV'|+\sum_{\mV'\in \mV'}n_{\mV'}>0\ ,
\]
and in this case this positive number is precisely the multiplicity of the eigenvalue $-2$.
\end{lemma}

All well-posedness results of this paper extend to generalized line graphs.

Generalized line graphs are not the only graphs that have adjacency matrices with eigenvalues not smaller than $-2$: the counterexamples are usually referred to as \textit{exceptional graphs} and 
 in the last two decades graphs with these properties have been studied in details and finally characterized, see e.g.~\cite{CveRowSim04}. The 	Petersen graph is a well-known exceptional graph.

\subsection{Adjacency-type operators on quantum graphs}

It is well-known that $\mathcal D^{-1}\mathcal A$ is the transition matrix associated with a random walk on the graph and it has been shown in~\cite{CarWoe07} that an interesting interplay exists between the spectral properties of ${\mathcal D}^{-1}\mathcal A$ and those of a certain bounded, self-adjoint operator $A$ over $L^2(\mathfrak G)$, where $\mathfrak G$ is the \textit{quantum graph} associated with $\mG$. Such an $A$ can be interpreted as the operator that maps functions over a quantum graph $\mathfrak G$ into their average over balls of radius 1 with respect to the canonical structure of a quantum graph as a metric measure space. Spectral relations between ${\mathcal D}^{-1}\mathcal A$ and $A$ as well as a representation of $A$ in terms of the quantum graph Laplacian via self-adjoint functional calculus have been proved in~\cite{CarWoe07,LenPan16}.
 The considerations in~\cite{CarWoe07} suggest that the correct quantum graph counterpart of our (unnormalized!) adjacency matrices is instead a linear combination $B$ of Volterra operators: such an operator $B$ then maps functions over $\mathfrak G$ into their integrals over balls of radius 1. Following the method suggested in~\cite[\S~14.4]{Haa14} we find that $\|B\|=\frac{4}{\pi}$ in case of the quantum graph built upon $\Z$ and, more generally, $\|B\|\asymp \deg_{\max}(\mG)$.$  $

\section{Appendix: Line graphs and their adjacency matrix}\label{sec:append}

Throughout this paper we discuss properties of the adjacency matrix $\mathcal A$ of a given graph $\mG$, under our standing assumptions that $\mG$ is a locally finite, simple graph with nonempty vertex and edge sets $\mV$ and $\mE$, respectively. Our most frequent setting is that $
\mG$ is the line graph of a connected graph $\mH=(\mV',\mE')$.
(In order to avoid trivialities, we will always assume that $\mE\neq\emptyset$, i.e., $\mH$ has at least three vertices.)

In agreement with the graph theoretical literature, see e.g.~\cite{Die05}, we mean by this that $\mG$ is constructed from $\mH$ according to the following procedure:
\begin{itemize}
\item $\mV=\mE'$, i.e., each edge of $\mH$ defines a vertex of $\mG$;
\item two vertices $\mv\simeq \me',\mw\simeq \mf'\in \mV$ are adjacent by means of an edge if and only if the corresponding edges $\me',\mf'\in \mE'$ share an endpoint.
\end{itemize}

Under these assumptions, we can consider an arbitrary orientation of the graph $\mH$ and the associated \emph{oriented} incidence matrix $\mathcal I:=({\iota}_{\mv' \me'})$ defined for all vertices $\mv'\in \mV'$ and all edges $\me'\in \mE'$ of $\mH$ by
$${\iota}_{\mv' \me'}:=\left\{
\begin{array}{ll}
-1 & \hbox{if } \mv' \hbox{ is initial endpoint of } \me', \\
+1 & \hbox{if } \mv' \hbox{ is terminal endpoint of } \me', \\
0 & \hbox{otherwise}
\end{array}\right.$$
as well as the signless incidence matrix $\mathcal J:=(|{\iota}_{\mv' \me'}|)$, i.e.,
\begin{equation}\label{eq:def-signlessinc}
|{\iota}_{\mv' \me'}|=\left\{
\begin{array}{ll}
+1 & \hbox{if } \mv' \hbox{ is an endpoint of } \me', \\
0 & \hbox{otherwise}
\end{array}\right.
\end{equation}
An edge $\me'\in \mE'$ is said to be \emph{incident} in a vertex $\mv'\in \mV'$ if $\iota_{\mv\me}\neq 0$. (In this paper, our focus lies on non-oriented graph: accordingly, it is mostly $\mathcal J$ we have used.)

Now, a direct computation shows that
\begin{equation}\label{eq:variationmain}
\mathcal A=\mathcal J^T \mathcal J-2\Id_{\mV}\ ,
\end{equation}
cf.~\cite[\S~1.1]{CveRowSim04},
where $\Id_{\mV}$ is the identity operator on the space of functions defined over $\mV\simeq \mE'$.

It is easy to see that if $\mG=(\mV,\mE)$ is the line graph of $\mH=(\mV',\mE')$, then its edge set $\mE$ has cardinality $|\mE|$ equal to
\[
\frac12 (\mathcal A{\bf 1}|{\bf 1})_{\ell^2(\mV)}=\frac12 \sum_{\mv'\in \mV'}\left|\sum_{\me'\in \mE'}|\iota_{\mv'\me'}| \right|^2 - |\mV'|\ .
\]
In particular, if a graph is finite (resp., countably infinite), then so is its line graph.

Not every simple graph $\mG$ is the line graph of some $\mH$: the claw graph -- i.e., the star on three edges $\mS_3$ -- is an example of a graph that is not a line graph; accordingly, a necessary condition for a graph to be a line graph is to be claw-free, i.e., not to contain $\mS_3$ as an induced subgraph. Indeed, Beineke's Theorem states that a finite simple graph is a line graph if and only if it does not contain any of nine forbidden graphs as induced subgraphs~\cite[Thm.~8.4]{Har69}. These include of course $\mS_3$ as well as the wheel $\mW_5$.

Nor is the graph $\mH$ uniquely determined by $\mG$:  the triangle $\mC_3$ is the line graph of both $\mS_3$ and $\mC_3$; however, there are no further pairs of graphs having the same line graph~\cite[Thm.~8.3]{Har69}. For this reason, we sometimes refer to $\mH$ as \textit{the} pre-line graph of $\mG$.

\bibliographystyle{alpha}
\newcommand{\etalchar}[1]{$^{#1}$}

\end{document}